\documentclass[a4paper,10pt]{amsart}
\usepackage{amssymb,amscd,amscd,tikz-cd}

\newtheorem{theorem}{Theorem}[section]
\newtheorem{lemma}[theorem]{Lemma}

\newtheorem{definition}[theorem]{Definition}
\newtheorem{example}[theorem]{Example}
\newtheorem{proposition}[theorem]{Proposition}
\newtheorem{corollary}[theorem]{Corollary}
\newtheorem{remark}[theorem]{Remark}

\begin{document}
\title[Valuation Ideal Factorization Domains]{Valuation Ideal Factorization Domains}

\author[G. W. Chang]{Gyu Whan Chang}
\address{\hspace*{-2.2mm} Department of Mathematics Education, Incheon National University, Incheon 22012, Korea}
\email{whan@inu.ac.kr}

\author[A. Reinhart]{Andreas Reinhart}
\address{\hspace*{-2.2mm} Institut f\"ur Mathematik und Wissenschaftliches Rechnen, Karl-Franzens-Universit\"at Graz, NAWI Graz, Heinrichstra{\ss}e 36, 8010 Graz, Austria}
\email{andreas.reinhart@uni-graz.at}

\thanks{Chang was supported by Basic Science Research Program through the National Research Foundation of Korea (NRF) funded by the Ministry of Education (2017R1D1A1B06029867). Reinhart was supported by the Austrian Science Fund FWF, Project Number P36852-N}
\date{\today}

\begin{abstract}
An integral domain $D$ is a {\em valuation ideal factorization domain} (VIFD) if each nonzero principal ideal of $D$ can be written as a finite product of valuation ideals. Clearly, $\pi$-domains are VIFDs. We study the ring-theoretic properties of VIFDs and the $*$-operation analogs of VIFDs. Among them, we show that if $D$ is treed (resp., $*$-treed), then $D$ is a VIFD (resp., $*$-VIFD) if and only if $D$ is an ${\rm h}$-local Pr\"ufer domain (resp., a $*$-${\rm h}$-local P$*$MD) if and only if every nonzero prime ideal of $D$ contains an invertible (resp., a $*$-invertible) valuation ideal. We also study integral domains $D$ such that for each nonzero nonunit $a\in D$, there is a positive integer $n$ such that $a^n$ can be written as a finite product of valuation elements.
\end{abstract}

\maketitle
\allowdisplaybreaks

\section{Introduction}

All rings considered in this paper are commutative with identity. Let $D$ be an integral domain with quotient field $K$. An {\em overring} of $D$ means a subring of $K$ containing $D$. As in \cite[Appendix 3]{zs60}, we say that an ideal $I$ of $D$ is a {\em valuation ideal} if there is a valuation overring $V$ of $D$ such that $IV\cap D=I$. Clearly, each ideal of a valuation domain is a valuation ideal. Conversely, in \cite[Corollary 2.4]{go65}, Gilmer and Ohm showed that if every principal ideal of $D$ is a valuation ideal, then $D$ is a valuation domain. Following \cite{cr20}, we say that a nonzero nonunit $a\in D$ is a {\em valuation element} if $aD$ is a valuation ideal, i.e., there is a valuation overring $V$ of $D$ such that $aV\cap D=aD$. In \cite{cr20}, we studied some ring-theoretic properties of {\em valuation factorization domains} (VFDs), which are integral domains whose nonzero nonunits can be written as a finite product of valuation elements. It is clear that valuation domains and UFDs are VFDs. In this paper, we continue our work on ideal factorization properties of integral domains. It is well-known that $D$ is a Dedekind domain (resp., $\pi$-domain) if and only if every nonzero ideal (resp., nonzero principal ideal) of $D$ can be written as a finite product of prime ideals; in particular, a Dedekind domain is a $\pi$-domain. Recall that a prime ideal of $D$ is a valuation ideal \cite[page 341]{zs60}, so a Dedekind domain (resp., $\pi$-domain) $D$ has the property that every nonzero principal ideal of $D$ can be written as a finite product of valuation ideals, which will be called a {\em valuation ideal factorization domain} (VIFD). In this paper, we study some ideal-theoretic properties of VFDs and VIFDs, and the $*$-operation analog of a VIFD, which is called a {\em $*$-VIFD}.

\smallskip

This paper consists of five sections including the introduction where we review the definitions related with the $t$- and $w$-operations. First, in Section 2, we investigate valuation ideals, valuation elements and their connections to star operations. Let $I$ be a proper valuation ideal of $D$. Among others, we show that if $I$ is $*$-invertible, then $I_*$ is a valuation ideal of $D$ and it is contained in a unique maximal $t$-ideal of $D$. In Section 3, we study the $*$-VIFDs. We introduce the notion of $*$-VIFDs and we show that all of these rings are integrally closed weakly Matlis domains. Furthermore, we prove that $D$ is a VFD if and only if $D$ is a VIFD with ${\rm Pic}(D)=\{0\}$, if and only if $D$ is a $*$-VIFD with ${\rm Cl}_*(D)=\{0\}$. It is also shown that $D$ is a $*$-${\rm h}$-local P$*$MD if and only if every nonzero ($*$-finite) $*$-ideal of $D$ can be written as a finite $*$-product of $*$-comaximal valuation ideals of $D$. We further study the $*$-VIFDs that are $*$-treed. In Section 4, we consider VIFDs and $t$-VIFDs. Finally, in Section 5, we investigate AVFDs, i.e., integral domains for which every nonzero nonunit has a power that is a finite product of valuation elements. Among other things, we show that if $D$ is a $*$-VIFD such that ${\rm Cl}_*(D)$ is a torsion group, then $D$ is an AVFD.

\smallskip

Throughout this paper, let $\mathbb{N}$, $\mathbb{N}_0$ and $\mathbb{Z}$ denote the set of all positive integers, the set of all nonnegative integers and the set of all integers, respectively. For $x,y\in\mathbb{Z}$ with $x\leq y$, let $[x,y]=\{z\in\mathbb{Z}\mid x\leq z\leq y\}$.

\subsection{Definitions related to star operations}

We first review some definitions related to the $t$-operation which are needed for fully understanding this paper. Let $D$ be an integral domain with quotient field $K$. A $D$-submodule $A$ of $K$ is called a {\em fractional ideal} of $D$ if $dA\subseteq D$ for some nonzero $d\in D$. An (integral) ideal of $D$ is a fractional ideal of $D$ that is contained in $D$. Let $F(D)$ (resp., $f(D)$) be the set of nonzero fractional (resp., nonzero finitely generated fractional) ideals of $D$. For $A,B\in F(D)$, let $(A:B)=\{x\in K\mid xB\subseteq A\}$ and $A^{-1}=(D:A)$. Observe that $(A:B)\in F(D)$ and $A^{-1}\in F(D)$.

\smallskip

Recall that a map $*:F(D)\rightarrow F(D), I\mapsto I_*$ is called a {\em star operation} on $D$ if the following conditions are satisfied for all $A,B\in F(D)$ and nonzero $c\in K$:

\begin{itemize}
\item $A\subseteq A_*=(A_*)_*$.
\item If $A\subseteq B$, then $A_*\subseteq B_*$.
\item $(cA)_*=cA_*$.
\item $D_*=D$.
\end{itemize}

\noindent
Let $*$ be a star operation on $D$. We say that $*$ is of {\em finite type} if for each $A\in F(D)$, $A_*=\bigcup_{C\in f(D),C\subseteq A} C_*$. Besides that, $*$ is said to be {\em stable} if $(A\cap B)_*=A_*\cap B_*$ for all $A,B\in F(D)$. If $*_1$ and $*_2$ are star operations on $D$, we mean by $*_1\leq *_2$ that $I_{*_1}\subseteq I_{*_2}$ for any $I\in F(D)$.

\smallskip

Let $*_f:F(D)\rightarrow F(D)$ be defined by $A_{*_f}=\bigcup_{C\in f(D),C\subseteq A} C_*$ for each $A\in F(D)$. Then $*_f$ is a star operation of finite type on $D$. Let $\widetilde{*}:F(D)\rightarrow F(D)$ be defined by $A_{\widetilde{*}}=\{x\in K\mid xJ\subseteq A$ for some $J\in f(D)$ with $J_*=D\}$ for each $A\in F(D)$. Then $\widetilde{*}$ is a stable star operation of finite type on $D$ \cite[Example 2.1 and Remark 2.3]{fjs03}. If we set

\begin{itemize}
\item $A_v=(A^{-1})^{-1}$,
\item $A_t=\bigcup\{I_v\mid I\subseteq A$ and $I\in f(D)\}$,
\item $A_w=\{x\in K\mid xJ\subseteq A$ for some $J\in f(D)$ with $J_v=D\}$, and
\item $A_d=A$
\end{itemize}

\noindent
for each $A\in F(D)$, then $v$ is a star operation on $D$, $t=v_f$, $w=\widetilde{v}$, and $d$ is a star operation on $D$ such that $d=d_f=\widetilde{d}$. It is known that $d\leq *\leq v$, $d\leq\widetilde{*}\leq *_f\leq t\leq v$, and $\widetilde{*}\leq w$ for any star operation $*$ on $D$.

\smallskip

An $I\in F(D)$ is called a {\em fractional $*$-ideal} of $D$ if $I_*=I$ and a fractional $*$-ideal $I$ of $D$ is called a {\em $*$-ideal} if $I\subseteq D$. A $*$-ideal is a {\em maximal $*$-ideal} if it is maximal among the proper $*$-ideals. Let $*$-${\rm Max}(D)$ be the set of maximal $*$-ideals of $D$ and let $*$-${\rm Spec}(D)$ be the set of prime $*$-ideals of $D$. It may happen that $*$-${\rm Max}(D)=\emptyset$ even though $D$ is not a field; for example, if $D$ is a rank-one nondiscrete valuation domain and $*=v$, then $*$-${\rm Max}(D)=\emptyset$. However, $*_f$-${\rm Max}(D)\neq\emptyset$ if and only if $D$ is not a field; each maximal $*_f$-ideal of $D$ is a prime ideal; $*_f$-${\rm Max}(D)\subseteq *_f$-${\rm Spec}(D)$; each proper $*_f$-ideal of $D$ is contained in a maximal $*_f$-ideal. Each prime ideal of $D$ minimal over a $*_f$-ideal is a $*_f$-ideal, whence each height-one prime ideal is a $*_f$-ideal; $D=\bigcap_{P\in *_f\textnormal{-}{\rm Max}(D)} D_P$; $I_{\widetilde{*}}=\bigcap_{M\in *_f\textnormal{-}{\rm Max}(D)} I_P$ for all $I\in F(D)$; and $*_f$-${\rm Max}(D)=\widetilde{*}$-${\rm Max}(D)$; see, for example, \cite[Lemma 2.1(2)]{or20}.

\smallskip

An integral domain $D$ is said to be of {\em finite $*$-character} if each nonzero nonunit of $D$ is contained in only finitely many maximal $*$-ideals. Recall from \cite{xaz22} that $D$ is a {\em $*$-${\rm h}$-local domain} if $D$ is of finite $*$-character and every nonzero prime $*$-ideal of $D$ is contained in a unique maximal $*$-ideal of $D$. Note that $D$ is $t$-${\rm h}$-local if and only if $D$ is weakly Matlis \cite[page 8]{az19} and $D$ is $d$-${\rm h}$-local if and only if $D$ is ${\rm h}$-local \cite[page 8]{az19}. The $*$-dimension of $D$ is defined by $*$-$\dim(D)=\sup\{n\in\mathbb{N}_0\mid P_1\subsetneq\cdots\subsetneq P_n$ for some prime $*$-ideals $P_i$ of $D\}$. Hence, $*$-$\dim(D)=1$ if and only if $D$ is not a field and each prime $*$-ideal of $D$ is a maximal $*$-ideal.

\smallskip

An $I\in F(D)$ is said to be {\em $*$-invertible} if $(II^{-1})_*=D$ and a fractional $*$-ideal $I$ of $D$ is said to be {\em $*$-finite} if $I=J_*$ for some $J\in f(D)$. We say that $D$ is a {\em Pr\"ufer $*$-multiplication domain} (P$*$MD) if each nonzero finitely generated ideal of $D$ is $*_f$-invertible. It is known that $D$ is a P$*$MD if and only if $D_P$ is a valuation domain for all maximal $*_f$-ideals $P$ of $D$ \cite[Theorem 3.1]{fjs03}. A Pr\"ufer domain is a P$v$MD whose maximal ideals are $t$-ideals. Let $T_*(D)$ be the set of $*$-invertible fractional $*$-ideals. Then $T_*(D)$ is an abelian group under $I*J=(IJ)_*$. Let ${\rm Inv}(D)$ (resp., ${\rm Prin}(D)$) be the subgroup of $T(D)$ of invertible (resp., nonzero principal) fractional ideals of $D$. The factor group ${\rm Cl}_*(D)=T_*(D)/{\rm Prin}(D)$, called the {\em $*$-class group} of $D$, is an abelian group and ${\rm Pic}(D)={\rm Inv}(D)/{\rm Prin}(D)$, called the {\em Picard group} of $D$, is a subgroup of ${\rm Cl}_*(D)$.

\smallskip

Let $S$ be a subset of $D$. Then $S$ is called {\em multiplicatively closed} if $1\in S$ and $xy\in S$ for all $x,y\in S$. It is clear that if $S$ is a multiplicatively closed set with $0\in S$, then $D_S\simeq\{0\}$, so we always assume that $0\not\in S$. Furthermore, $S$ is said to be {\em divisor-closed} or {\em saturated} if for all $x,y\in D$ with $xy\in S$, it follows that $x,y\in S$. If $a,b\in D$, then we write $a\mid_D b$ if there is some $c\in D$ such that $b=ac$. Hence, $S$ is divisor-closed if and only if $a\mid_D b$ implies $a\in S$ for any $b\in S$ and $a\in D$.

\smallskip

Let $D[X]$ be the polynomial ring over $D$. For a polynomial $f\in D[X]$, let $c(f)$ denote the ideal of $D$ generated by the coefficients of $f$. Let $N_v=\{f\in D[X]\mid f\neq 0$ and $c(f)_v=D\}$. Then, by the Dedekind-Mertens lemma, it can be shown that $N_v$ is a multiplicatively closed and divisor-closed subset of $D[X]$.

\subsection{Importance of star operations}

Next we give a short explanation for considering arbitrary star operations (of finite type) rather than the $d$-operation, the $t$-operation and the $w$-operation. First of all, star operations enable us to unify results that hold for all of the aforementioned operations. Besides that, we will see later that certain results are specific to the $t$-operation/$w$-operation, while other results hold for arbitrary star operations of finite type. We also emphasize that similar types of concepts have been investigated for star operations \cite{az19,xaz22,z21}.

\smallskip

If not stated otherwise, then from now on $D$ is always an integral domain with quotient field $K$ and $*$ is a star operation of finite type on $D$.

\section{Valuation ideals and valuation elements}

We first discuss briefly how to extend the star operation $*$ on $D$ to arbitrary nonzero $D$-submodules of $K$. For a nonzero $D$-submodule $A$ of $K$, let $A_*=\bigcup_{J\in F(D),J\subseteq A} J_*$, then $A_*=\bigcup_{J\in f(D),J\subseteq A} J_*$ because $*$ is of finite type, and hence $A_*$ is well-defined. Next we present a few simple properties of the $*$-closure of $D$-submodules of $K$ that we will use without further mention.

\begin{lemma}\label{Lemma 2.1}
Let $D$ be an integral domain with quotient field $K$, let $*$ be a star operation of finite type on $D$, let $c\in K$ be nonzero and let $A$ and $B$ be nonzero $D$-submodules of $K$. Then $A_*$ is a $D$-submodule of $K$, $A\subseteq A_*$, $cA_*=(cA)_*$, if $A\subseteq B_*$, then $A_*\subseteq B_*$, and $(AB)_*=(A_*B)_*$. Moreover, if $*$ is stable, then $(A\cap B)_*=A_*\cap B_*$.
\end{lemma}

\begin{proof}
This is well-known and straightforward to prove.
\end{proof}

Let $*$ be a star operation of finite type on $D$. Following \cite{e19}, we say that an overring $D^{\prime}$ of $D$ is {\em $*$-closed} if $(D^{\prime})_*=D^{\prime}$.

\begin{lemma}\label{Lemma 2.2}
Let $D$ be an integral domain, let $*$ be a stable star operation of finite type on $D$ and let $I$ be a valuation $*$-ideal of $D$. Then there exists a $*$-closed valuation overring $W$ of $D$ such that $IW\cap D=I$.
\end{lemma}

\begin{proof}
There exists a valuation overring $V$ of $D$ such that $IV\cap D=I$. Set $W=V_*$. Observe that $W$ is a $*$-closed overring of $V$. Consequently, $W$ is a valuation overring of $D$ (since $W$ is an overring of a valuation domain) and $I\subseteq IW\cap D\subseteq (IW)_*\cap D=(IV)_*\cap D_*=(IV\cap D)_*=I_*=I$ by Lemma~\ref{Lemma 2.1}. Thus, $IW\cap D=I$.
\end{proof}

Let $D^{\times}$ denote the {\em unit group} of $D$ and let ${\rm Spec}(D)$ denote the {\em set of prime ideals} of $D$. For each ideal $I$ of $D$, let $\mathcal{P}(I)$ be the {\em set of minimal prime ideals} of $I$ (i.e., the set of prime ideals of $D$ that are minimal over $I$) and let $\sqrt{I}=\sqrt[D]{I}=\bigcap_{P\in {\rm Spec}(D),I\subseteq P} P=\{x\in D\mid x^n\in I$ for some $n\in\mathbb{N}\}$ be the {\em radical} of $I$.

\begin{proposition}\label{Proposition 2.3}
Let $D$ be an integral domain that is not a field, let $*$ be a star operation of finite type on $D$ and let $I$ be a valuation $*$-ideal of $D$\textnormal{:}
\begin{enumerate}
\item There exists some $M\in *$-${\rm Max}(D)$ such that $I_M\cap D=I$.
\item If $I$ is proper, then $\sqrt{I}$ is a prime $*$-ideal of $D$.
\end{enumerate}
\end{proposition}

\begin{proof}
Observe that $\widetilde{*}$ is a stable star operation of finite type such that $J_{\widetilde{*}}\subseteq J_*$ for each $J\in F(D)$ and $*$-${\rm Max}(D)=\widetilde{*}$-${\rm Max}(D)$. Hence, $I$ is a valuation $\widetilde{*}$-ideal of $D$. Thus, $IW\cap D=I$ for some $\widetilde{*}$-closed valuation overring $W$ of $D$ by Lemma~\ref{Lemma 2.2}.

(1) Set $P=(W\setminus W^{\times})\cap D$. We show that $P$ is a prime $\widetilde{*}$-ideal of $D$. Clearly, $P$ is a prime ideal of $D$ (since $W\setminus W^{\times}$ is a prime ideal of $W$). Since $W$ is $\widetilde{*}$-closed, it is an easy consequence of Lemma~\ref{Lemma 2.1} that $*_1:F(W)\rightarrow F(W)$ defined by $A_{*_1}=A_{\widetilde{*}}$ for each $A\in F(W)$ is a star operation on $W$. It is straightforward to prove that $\ast_1$ is of finite type. Since $W$ is a valuation domain, the only star operation of finite type on $W$ is the $d$-operation, and thus $A_{\widetilde{*}}=A_{*_1}=A$ for each $A\in F(W)$. It follows that $W\setminus W^{\times}$ is $\widetilde{*}$-closed, and hence $P$ is a $\widetilde{*}$-ideal of $D$.

Since $P$ is a prime $\widetilde{*}$-ideal of $D$ and $\widetilde{*}$ is of finite type, there exists some $M\in\widetilde{*}$-${\rm Max}(D)$ such that $P\subseteq M$. Observe that $M\in *$-${\rm Max}(D)$. Since $D\setminus M\subseteq D\setminus P=D\setminus (W\setminus W^{\times})=W^{\times}\cap D\subseteq W^{\times}$, we have that $D_M\subseteq D_P\subseteq W$. Consequently, $I\subseteq I_M\cap D\subseteq IW\cap D=I$, and hence $I_M\cap D=I$.

(2) Let $I$ be proper. Since $*$ is of finite type, it is clear that $\sqrt{I}$ is a $*$-ideal of $D$. Since $W$ is a valuation domain, we have that $\sqrt{IW}$ is a prime ideal of $W$. Since $\sqrt{I}=\sqrt{IW\cap D}=\sqrt{IW}\cap D$, we infer that $\sqrt{I}$ is a prime ideal of $D$.
\end{proof}

\begin{corollary}\label{Corollary 2.4}
Let $D$ be an integral domain and let $I$ be a proper $t$-invertible valuation $t$-ideal of $D$. Then there exists a unique $M\in t$-${\rm Max}(D)$ such that $I\subseteq M$.
\end{corollary}

\begin{proof}
This is an immediate consequence of Proposition~\ref{Proposition 2.3}(1) and \cite[Lemma 4.2]{amz92}.
\end{proof}

The next example shows that the $t$-operation in Corollary~\ref{Corollary 2.4} cannot be replaced by an arbitrary star operation of finite type.

\begin{example}\label{Example 2.5}
{\em Let $D=\mathbb{Z}[X]$ be the polynomial ring over $\mathbb{Z}$ and let $\mathbb{P}$ be the set of prime numbers. Then $D$ is a two-dimensional Noetherian UFD and $X$ is a valuation element of $D$ $($since $X$ is a prime element of $D)$. Moreover, the prime ideal $XD$ is contained in infinitely many maximal ideals of $D$ $($since $(pD+XD)_{p\in\mathbb{P}}$ is a sequence of distinct maximal ideals of $D$ that contain $XD)$. In particular, $D$ is not of finite character and there exists a nonzero prime ideal of $D$ that is contained in more than one maximal ideal of $D$.}
\end{example}

A nonzero nonunit $a\in D$ is said to be {\em homogeneous} if it is contained in a unique maximal $t$-ideal of $D$. Following \cite{c21}, we say that $D$ is a {\em homogeneous factorization domain} (HoFD) if every nonzero nonunit of $D$ can be written as a finite product of homogeneous elements. Next we want to emphasize that various notions were already introduced in the literature that can be used to define the principal ideals generated by a homogeneous element. Let $a\in D$ be a nonzero nonunit. Then $a$ is homogeneous if and only if $aD$ is a unidirectional $t_{t\textnormal{-}{\rm Max}(D)}$-ideal in the sense of \cite{az99} if and only if $aD$ is a $t$-rigid ideal in the sense of \cite{hz20}.

\begin{corollary}\label{Corollary 2.6}
Let $D$ be an integral domain\textnormal{:}
\begin{enumerate}
\item If $a\in D$ is a valuation element, then $a$ is homogeneous.
\item If $D$ is a VFD, then $D$ is an integrally closed HoFD.
\end{enumerate}
\end{corollary}

\begin{proof}
(1) Observe that $aD$ is both a valuation ideal and a $t$-invertible $t$-ideal. Hence, by Corollary~\ref{Corollary 2.4}, $aD$ is contained in a unique maximal $t$-ideal of $D$, and thus $a$ is homogeneous.

(2) A VFD is integrally closed \cite[Corollary 1.5]{cr20}. Moreover, a valuation element is homogeneous by (1). Thus, $D$ is an integrally closed HoFD.
\end{proof}

Let $D$ be a quasi-local domain of dimension one that is not a valuation domain. Then every nonzero element of $D$ is homogeneous, while $D$ is not a VFD and $D$ has no valuation element \cite[Corollary 1.6]{cr20}. Therefore, a homogeneous element need not be a valuation element and an HoFD need not be a VFD in general.

\begin{remark}\label{Remark 2.7}
{\em We want to emphasize that valuation elements are not just homogeneous elements, but they also satisfy another interesting property that was studied by Zafrullah. Following \cite{z75}, we say that a nonzero nonunit $a\in D$ is {\em rigid} if for all $b,c\in D$ with $b\mid_D a$ and $c\mid_D a$, it follows that $b\mid_D c$ or $c\mid_D b$. Moreover, $D$ is called {\em semirigid} if every nonzero nonunit of $D$ is a finite product of rigid elements. Note that every valuation element is rigid \cite[Corollary 1.2(2)]{cr20}, and hence every VFD is semirigid. For more information on rigid elements and semirigid domains, we refer to \cite{z75,z77,z22}.}
\end{remark}

In what follows we provide connections to well-known types of elements. Let $u\in D$ be a nonzero nonunit. Then $u$ is called an {\em atom} of $D$ if for all $a,b\in D$ with $u=ab$, either $a$ is a unit of $D$ or $b$ is a unit of $D$. We say that $u$ is {\em primary} if $uD$ is a primary ideal of $D$. Furthermore, $D$ is called {\em atomic} if every nonzero nonunit of $D$ is a finite product of atoms of $D$. Observe that every atom is rigid, so every atomic domain is a semirigid domain. (Let $u\in D$ be an atom of $D$ and let $a,b\in D$ be such that $a\mid_D u$ and $b\mid_D u$. If $a\in D^{\times}$, then clearly $a\mid_D b$. Now let $a\not\in D^{\times}$. There is some $c\in D$ such that $u=ac$. We infer that $c\in D^{\times}$, and hence $b\mid_D u\mid_D uc^{-1}=a$.) It is known that $D$ is a UFD if and only if $D$ is an atomic VFD \cite[Corollary 2.4]{cr20}. Also note that every primary element is homogeneous \cite[Lemma 2.1]{acp03}. For the sake of clarity, we provide the following diagram to visualize the relations between the various types of elements.

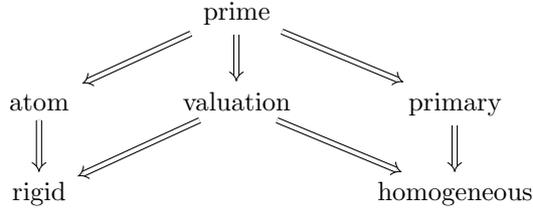
\[
\begin{tikzcd}
& \textnormal{prime}\arrow[ld,Rightarrow]\arrow[d,Rightarrow]\arrow[rd,Rightarrow] &\\
\textnormal{\quad atom\quad}\arrow[d,Rightarrow] & \textnormal{valuation}\arrow[ld,Rightarrow]\arrow[rd,Rightarrow] & \textnormal{primary}\arrow[d,Rightarrow]\\
\textnormal{rigid} & & \textnormal{homogeneous}
\end{tikzcd}
\]

\medskip

In general, a primary atom need not be a valuation element. (Let $D$ be an atomic quasi-local one-dimensional domain that is not a valuation domain and let $v\in D$ be an atom. Then $v$ is a primary element of $D$ but not a valuation element of $D$ \cite[Corollary 1.6]{cr20}.) We also want to mention that a valuation element is in general neither primary nor an atom. (Let $V$ be a two-dimensional valuation domain, let $P$ be the unique height-one prime ideal of $V$ and let $x\in P$ be nonzero. Then $x$ is a valuation element of $V$ but $x$ is neither primary nor an atom.) Note that an atom does not have to be homogeneous. (Let $D=\mathbb{Z}[\sqrt{10}]$. Then $3$ is an atom of $D$ and $3$ is not homogeneous.) Finally, a primary element need not be rigid. (Let $D$ be a quasi-local one-dimensional domain that is not a valuation domain. Then there is some nonzero nonunit $y\in D$ such that $y$ is not rigid (e.g., see \cite{z22}) and yet $y$ is primary.) In particular, we obtain that none of the implications in the diagram above can be reversed.

\medskip

Let $I$ be a $*$-ideal of $D$. We say that $I$ is {\em $*$-locally principal} if $I_M$ is a principal ideal of $D_M$ for each $M\in *$-${\rm Max}(D)$. It is easy to see that a $*$-invertible $*$-ideal of $D$ is a $*$-locally principal ideal for any star operation $*$ of finite type on $D$ \cite[12.3 Theorem]{h98}. We are going to give a $*$-locally principal ideal analog of \cite[Proposition 1.1]{cr20} that, among other things, shows that if $aR\cap D=aD$ and $bR\cap D=bD$, then $abR\cap D=abD$ for any $a,b\in D$ and an overring $R$ of $D$. We first need a lemma which is also a natural generalization of \cite[Proposition 1.1]{cr20}.

\begin{lemma}\label{Lemma 2.8}
Let $D$ be an integral domain, let $R$ be an overring of $D$, let $I$ be a nonzero principal ideal of $D$ and let $J$ be an ideal of $D$\textnormal{:}
\begin{enumerate}
\item If $IR\cap D=I$ and $JR\cap D=J$, then $(IJ)R\cap D=IJ$.
\item If $(IJ)R\cap D=IJ$, then $JR\cap D=J$.
\end{enumerate}
\end{lemma}

\begin{proof}
(1) Let $IR\cap D=I$ and $JR\cap D=J$. It suffices to show that $(IJ)R\cap D\subseteq IJ$. Let $x\in (IJ)R\cap D$. Since $IJ\subseteq I$, we infer that $x\in IR\cap D=I$, and hence $xI^{-1}\subseteq II^{-1}=D$. Observe that $xI^{-1}\subseteq (IJ)RI^{-1}=((II^{-1})J)R=JR$. This implies that $xI^{-1}\subseteq JR\cap D=J$. Therefore, $x\in xD=x(II^{-1})=I(xI^{-1})\subseteq IJ$.

(2) Let $(IJ)R\cap D=IJ$. It suffices to show that $JR\cap D\subseteq J$. Let $x\in JR\cap D$. Then $xI\subseteq (IJ)R$. Since $I\subseteq D$, it follows that $xI\subseteq (IJ)R\cap D=IJ$, and hence $x\in xD=x(II^{-1})\subseteq (IJ)I^{-1}=(II^{-1})J=J$. Thus, $JR\cap D\subseteq J$.
\end{proof}

\begin{proposition}\label{Proposition 2.9}
Let $D$ be an integral domain, let $R$ be an overring of $D$, let $*$ be a star operation of finite type on $D$, let $I$ be a nonzero $*$-locally principal $*$-ideal of $D$ and let $J$ be a nonzero $*$-ideal of $D$\textnormal{:}
\begin{enumerate}
\item If $IR\cap D=I$ and $JR\cap D=J$, then $(IJ)_*R\cap D=(IJ)_*$.
\item If $(IJ)_*R\cap D=(IJ)_*$, then $JR\cap D=J$.
\item If $IR\cap D=I$, $\sqrt{I}\subseteq\sqrt{J}$ and $I$ and $J$ are $*$-invertible, then $JR\cap D=J$.
\end{enumerate}
In particular, $IR\cap D=I$ if and only if $(I^n)_*R\cap D=(I^n)_*$ for some $n\in\mathbb{N}$ if and only if $(I^n)_*R\cap D=(I^n)_*$ for each $n\in\mathbb{N}$.
\end{proposition}

\begin{proof}
First we show that $((IJ)_*)_M=I_MJ_M$ for each $M\in *$-${\rm Max}(D)$. Let $M\in *$-${\rm Max}(D)$ and let $x\in (IJ)_*$. There is some nonzero $y\in I$ such that $I_M=yD_M$. Observe that $x\in A_*$ for some nonzero finitely generated ideal $A$ of $D$ with $A\subseteq IJ\subseteq I_MJ_M=yJ_M$. We have that $bA\subseteq yJ$ for some $b\in D\setminus M$, and thus $x\in A_*=b^{-1}(bA)_*\subseteq b^{-1}(yJ)_*=b^{-1}yJ\subseteq yJ_M=I_MJ_M$. This implies that $(IJ)_*\subseteq I_MJ_M$, and hence $((IJ)_*)_M=I_MJ_M$.

(1) Let $IR\cap D=I$, let $JR\cap D=J$ and let $x\in (IJ)_*R\cap D$. Since $*$ is of finite type, it is sufficient to show that $x\in ((IJ)_*)_M$ for each $M\in *$-${\rm Max}(D)$. Let $M\in *$-${\rm Max}(D)$. Then $x\in ((IJ)_*R\cap D)_M=((IJ)_*)_MR_M\cap D_M=I_MJ_MR_M\cap D_M=I_MJ_M=((IJ)_*)_M$, where the third equality holds by Lemma~\ref{Lemma 2.8}(1) (since $I_M$ is a principal ideal of $D_M$, $I_M=(IR\cap D)_M=I_MR_M\cap D_M$, $J_M=(JR\cap D)_M=J_MR_M\cap D_M$ and $R_M$ is an overring of $D_M$).

(2) Let $(IJ)_*R\cap D=(IJ)_*$ and let $x\in JR\cap D$. Since $*$ is of finite type, it remains to show that $x\in J_M$ for each $M\in *$-${\rm Max}(D)$. Let $M\in *$-${\rm Max}(D)$. Then $x\in JR\cap D\subseteq (JR\cap D)_M=J_MR_M\cap D_M$ and $I_MJ_MR_M\cap D_M=((IJ)_*)_MR_M\cap D_M=((IJ)_*R\cap D)_M=((IJ)_*)_M=I_MJ_M$. Consequently, $x\in J_MR_M\cap D_M=J_M$ by Lemma~\ref{Lemma 2.8}(2) (since $I_M$ is a principal ideal and $R_M$ is an overring of $D_M$).

(3) Let $IR\cap D=I$, let $\sqrt{I}\subseteq\sqrt{J}$ and let $I$ and $J$ be $*$-invertible. Since $*$ is of finite type and $I$ is a $*$-invertible $*$-ideal, we have that $I$ is $*$-finite. Since $I\subseteq\sqrt{J}$ and $I$ is $*$-finite, we infer that $(I^n)_*\subseteq J$ for some $n\in\mathbb{N}$. It follows that $(I^n)_*=(JL)_*$ for some $*$-invertible $*$-ideal $L$ of $D$ (since $J$ is a $*$-invertible $*$-ideal of $D$). It follows from (1) above that $(I^n)_*R\cap D=(I^n)_*$, and hence $(JL)_*R\cap D=(JL)_*$. Therefore, $JR\cap D=J$ by (2) above.
\end{proof}

\begin{corollary}\label{Corollary 2.10}\cite[Corollary 1.2]{cr20}.
Let $D$ be an integral domain, let $*$ be a star operation of finite type on $D$, let $I$ be a proper $*$-invertible valuation $*$-ideal of $D$ and let $J$ be a $*$-invertible $*$-ideal of $D$\textnormal{:}
\begin{enumerate}
\item If $L$ is a $*$-invertible $*$-ideal of $D$ with $I\subseteq J\cap L$, then $J$ and $L$ are comparable.
\item If $\sqrt{I}\subseteq\sqrt{J}$, then $I$ and $J$ are comparable.
\item $\bigcap_{n\in\mathbb{N}} (I^n)_*$ is a prime $*$-ideal of $D$.
\item If $\sqrt{I}\subsetneq\sqrt{J}$, then $I\subseteq\bigcap_{n\in\mathbb{N}} (J^n)_*$.
\end{enumerate}
\end{corollary}

\begin{proof}
Note that the $*$-invertible $*$-ideals of $D$ are precisely the $\widetilde{*}$-invertible $\widetilde{*}$-ideals of $D$ \cite[Lemma 2.1(3)]{or20}. By Lemma~\ref{Lemma 2.2}, there is a $\widetilde{*}$-closed valuation overring $V$ of $D$ such that $IV\cap D=I$.

(1) Let $L$ be a $*$-invertible $*$-ideal of $D$ such that $I\subseteq J\cap L$. Then $JV\cap D=J$ and $LV\cap D=L$ by Proposition~\ref{Proposition 2.9}(3). Moreover, since $V$ is a valuation domain, $JV$ and $LV$ are comparable. Thus, $J$ and $L$ are comparable.

(2) Let $\sqrt{I}\subseteq\sqrt{J}$. Then $JV\cap D=J$ by Proposition~\ref{Proposition 2.9}(3). Since $IV$ and $JV$ are comparable, we have that $I$ and $J$ are comparable.

(3) Clearly, $\bigcap_{n\in\mathbb{N}} (I^n)_*$ is a $*$-ideal of $D$. Note that $(I^n)_*=(I^n)_{\widetilde{*}}$ for each $n\in\mathbb{N}$ by \cite[Lemma 2.1(3)]{or20}. There is a finitely generated ideal $L$ of $D$ with $I=L_{\widetilde{*}}$. Consequently, $LV=aV$ for some $a\in L$. Note that $aV=LV\subseteq IV\subseteq (IV)_{\widetilde{*}}=(LV)_{\widetilde{*}}=(aV)_{\widetilde{*}}=aV$, and hence $IV=aV$. Along similar lines, one can prove that $(I^n)_*V=a^nV$ for each $n\in\mathbb{N}$. It follows by Proposition~\ref{Proposition 2.9}(1) that $(I^n)_*V\cap D=(I^n)_*$ for each $n\in\mathbb{N}$. Note that $\bigcap_{n\in\mathbb{N}} a^nV$ is a prime ideal of $V$. Consequently, $\bigcap_{n\in\mathbb{N}} (I^n)_*=\bigcap_{n\in\mathbb{N}} ((I^n)_*V\cap D)=(\bigcap_{n\in\mathbb{N}} (I^n)_*V)\cap D$ is a prime ideal of $D$.

(4) Let $\sqrt{I}\subsetneq\sqrt{J}$ and let $n\in\mathbb{N}$. Then $\sqrt{I}\subseteq\sqrt{J}=\sqrt{(J^n)_*}$, and hence $I$ and $(J^n)_*$ are comparable by (2). If $(J^n)_*\subseteq I$, then $\sqrt{J}=\sqrt{(J^n)_*}\subseteq\sqrt{I}$, a contradiction. Therefore, $I\subseteq (J^n)_*$.
\end{proof}

In this paper we study integral domains $D$ in which each nonzero principal ideal can be written as a finite $*$-product of valuation $*$-ideals for a given star operation $*$ on $D$, and in this case, all of the valuation ideals in question must be $*$-invertible. The next result shows that such a finite $*$-product of $*$-invertible valuation $*$-ideals can be written in a specific form. For example, if $I$ is a finite $*$-product of $*$-invertible valuation $*$-ideals of $D$, say, $I=(\prod_{i=1}^n I_i)_*$ and each $I_i$ is a proper $*$-ideal of $D$, then $\mathcal{P}(I)\subseteq\{\sqrt{I_k}\mid k\in [1,n]\}$ and $n\geq |\mathcal{P}(I)|$.

\begin{proposition}\label{Proposition 2.11}
Let $D$ be an integral domain, let $*$ be a star operation of finite type on $D$ and let $I$ be a finite $*$-product of $*$-invertible valuation $*$-ideals of $D$. Then $|\mathcal{P}(I)|=\min\{m\in\mathbb{N}_0\mid I$ is a $*$-product of $m$ $*$-invertible valuation $*$-ideals of $D\}$ and there are $*$-invertible valuation $*$-ideals $(I(P))_{P\in\mathcal{P}(I)}$ of $D$ such that $I=(\prod_{P\in\mathcal{P}(I)} I(P))_*$ and $\sqrt{I(Q)}=Q$ for each $Q\in\mathcal{P}(I)$.
\end{proposition}

\begin{proof}
Without restriction we can assume that $I$ is proper. Let $n$ be the smallest positive integer such that $I$ is the $*$-product of $n$ $*$-invertible valuation $*$-ideals of $D$. Then there are $*$-invertible valuation $*$-ideals $I_i$ of $D$ such that $I=(\prod_{i=1}^n I_i)_*$. First we show that $\mathcal{P}(I)\subseteq\{\sqrt{I_i}\mid i\in [1,n]\}$. Let $P\in\mathcal{P}(I)$. Then $\prod_{i=1}^n I_i\subseteq I\subseteq P$, and hence there is some $j\in [1,n]$ such that $I_j\subseteq P$. We infer that $I\subseteq I_j\subseteq\sqrt{I_j}\subseteq P$. Since $\sqrt{I_j}$ is a prime ideal of $D$ by Proposition~\ref{Proposition 2.3}(2), it follows that $P=\sqrt{I_j}$.

This implies that $|\mathcal{P}(I)|\leq n$. Next we show that for each $P\in\mathcal{P}(I)$, $|\{i\in [1,n]\mid P\subseteq\sqrt{I_i}\}=|\{i\in [1,n]\mid P=\sqrt{I_i}\}|=1$. Let $P\in\mathcal{P}(I)$. Then $P=\sqrt{I_j}$ for some $j\in [1,n]$, and thus $\{i\in [1,n]\mid P\subseteq\sqrt{I_i}\}\supseteq\{i\in [1,n]\mid P=\sqrt{I_i}\}\not=\emptyset$. Set $\mathcal{I}=\{i\in [1,n]\mid P\subseteq\sqrt{I_i}\}$ and $J=(\prod_{i=1,P\subseteq\sqrt{I_i}}^n I_i)_*$. It is sufficient to show that $\mathcal{I}$ is a singleton. Note that $\sqrt{J}=\bigcap_{i=1,P\subseteq\sqrt{I_i}}^n\sqrt{I_i}=P$ (since $\{i\in [1,n]\mid P=\sqrt{I_i}\}\not=\emptyset$). Clearly, $J$ is a $*$-invertible $*$-ideal of $D$. Since $\sqrt{I_j}=P=\sqrt{J}$, we infer by Proposition~\ref{Proposition 2.9}(3) that $J$ is a valuation ideal of $D$. Because of the minimality of $n$, we have that $\mathcal{I}$ is a singleton.

Since $[1,n]\subseteq\bigcup_{Q\in\mathcal{P}(I)}\{i\in [1,n]\mid Q\subseteq\sqrt{I_i}\}$, it follows that $n\leq\sum_{Q\in\mathcal{P}(I)}|\{i\in [1,n]\mid Q\subseteq\sqrt{I_i}\}|=|\mathcal{P}(I)|$, and hence $|\mathcal{P}(I)|=n$. Let $f:[1,n]\rightarrow\mathcal{P}(I)$ be defined by $f(i)=\sqrt{I_i}$ for each $i\in [1,n]$. Then $f$ is a well-defined bijection. For each $Q\in\mathcal{P}(I)$, set $I(Q)=I_{f^{-1}(Q)}$. Then $I=(\prod_{Q\in\mathcal{P}(I)} I(Q))_*$ and $\sqrt{I(Q)}=Q$ for each $Q\in\mathcal{P}(I)$.
\end{proof}

We continue our investigation of ideals that are finite $*$-products of $*$-invertible valuation $*$-ideals. The next result serves as a preparatory result for Theorem~\ref{Theorem 3.8}, but it is more generally applicable, since it holds for arbitrary integral domains.

\begin{proposition}\label{Proposition 2.12}
Let $D$ be an integral domain, let $*$ be a star operation of finite type on $D$ and let $\Omega$ be the set of finite $*$-products of $*$-invertible valuation $*$-ideals of $D$. Then for all $I,J,L\in\Omega$ with $(JL)_*\subseteq I$, there are some $J^{\prime},J^{\prime\prime},L^{\prime},L^{\prime\prime}\in\Omega$ such that $I=(J^{\prime}L^{\prime})_*$, $J=(J^{\prime}J^{\prime\prime})_*$ and $L=(L^{\prime}L^{\prime\prime})_*$.
\end{proposition}

\begin{proof}
It is sufficient to show by induction that for each $m\in\mathbb{N}$ and all $I,J,L\in\Omega$ such that $I$ is a $*$-product of $m$ $*$-invertible valuation $*$-ideals of $D$ and $(JL)_*\subseteq I$, there are some $J^{\prime},J^{\prime\prime},L^{\prime},L^{\prime\prime}\in\Omega$ such that $J=(J^{\prime}J^{\prime\prime})_*$, $L=(L^{\prime}L^{\prime\prime})_*$ and $I=(J^{\prime}L^{\prime})_*$.

\medskip
First let $m=1$. By Proposition~\ref{Proposition 2.11}, it remains to show by induction that for each $n\in\mathbb{N}_0$, each $*$-invertible valuation $*$-ideal $I$ of $D$ and all $J,L\in\Omega$ with $|\mathcal{P}(J)|+|\mathcal{P}(L)|=n$ and $(JL)_*\subseteq I$, there are some $J^{\prime},J^{\prime\prime},L^{\prime},L^{\prime\prime}\in\Omega$ such that $J=(J^{\prime}J^{\prime\prime})_*$, $L=(L^{\prime}L^{\prime\prime})_*$ and $I=(J^{\prime}L^{\prime})_*$. Let $n\in\mathbb{N}_0$, let $I$ be a $*$-invertible valuation $*$-ideal of $D$ and let $J,L\in\Omega$ be such that $|\mathcal{P}(J)|+|\mathcal{P}(L)|=n$ and $(JL)_*\subseteq I$. Without restriction let $I$ be proper. Since $I$ is a valuation ideal of $D$, $\sqrt{I}$ is a prime ideal of $D$ by Proposition~\ref{Proposition 2.3}(2), and thus $J\subseteq\sqrt{I}$ or $L\subseteq\sqrt{I}$. Without restriction let $J\subseteq\sqrt{I}$. By Proposition~\ref{Proposition 2.11}, there are $*$-invertible valuation $*$-ideals $(J(P))_{P\in\mathcal{P}(J)}$ of $D$ such that $J=(\prod_{P\in\mathcal{P}(J)} J(P))_*$. Consequently, $J(Q)\subseteq\sqrt{I}$ for some $Q\in\mathcal{P}(J)$. This implies that $\sqrt{J(Q)}\subseteq\sqrt{I}$. It follows from Corollary~\ref{Corollary 2.10}(2) that $J(Q)$ and $I$ are comparable.

\medskip
{\noindent}{{\bf Case 1.}} $J(Q)\subseteq I$. Then $J(Q)=(IA)_*$ for some $*$-invertible $*$-ideal $A$ of $D$. It follows from Proposition~\ref{Proposition 2.9}(2) that $A$ is a valuation ideal of $D$. Set $J^{\prime}=I$, $J^{\prime\prime}=(A\prod_{P\in\mathcal{P}(J)\setminus\{Q\}} J(P))_*$, $L^{\prime}=D$ and $L^{\prime\prime}=L$. Then $J^{\prime},J^{\prime\prime},L^{\prime},L^{\prime\prime}\in\Omega$, $J=(J(Q)\prod_{P\in\mathcal{P}(J)\setminus\{Q\}} J(P))_*=(J^{\prime}J^{\prime\prime})_*$, $L=(L^{\prime}L^{\prime\prime})_*$ and $I=(J^{\prime}L^{\prime})_*$.

\medskip
{\noindent}{{\bf Case 2.}} $I\subsetneq J(Q)$. Then $I=(J(Q)C)_*$ for some $*$-invertible $*$-ideal $C$ of $D$. By Proposition~\ref{Proposition 2.9}(2), we infer that $C$ is a valuation ideal of $D$. Set $B=(\prod_{P\in\mathcal{P}(J)\setminus\{Q\}} J(P))_*$. Then $B\in\Omega$ and $(J(Q)BL)_*=(JL)_*\subseteq I=(J(Q)C)_*$. Since $J(Q)$ is $*$-invertible, it follows that $(BL)_*\subseteq C$. Moreover, $\mathcal{P}(B)=\mathcal{P}(J)\setminus\{Q\}$, and hence $|\mathcal{P}(B)|+|\mathcal{P}(L)|<n$. Therefore, there are $B^{\prime},B^{\prime\prime},L^{\prime},L^{\prime\prime}\in\Omega$ such that $B=(B^{\prime}B^{\prime\prime})_*$, $L=(L^{\prime}L^{\prime\prime})_*$ and $C=(B^{\prime}L^{\prime})_*$ by the induction hypothesis. Set $J^{\prime}=(J(Q)B^{\prime})_*$ and $J^{\prime\prime}=B^{\prime\prime}$. Then $J^{\prime},J^{\prime\prime}\in\Omega$, $J=(J(Q)B)_*=(J^{\prime}J^{\prime\prime})_*$ and $I=(J(Q)C)_*=(J^{\prime}L^{\prime})_*$. This concludes the proof of the base step.

\medskip
Now let $m\in\mathbb{N}$ and $I,J,L\in\Omega$ be such that $I$ is a $*$-product of $m+1$ $*$-invertible valuation $*$-ideals of $D$ and $(JL)_*\subseteq I$. Clearly, there are some $A,B\in\Omega$ such that $A$ is a $*$-product of $m$ $*$-invertible valuation $*$-ideals of $D$ and $B$ is a $*$-invertible valuation $*$-ideal of $D$ such that $I=(AB)_*$. Then $(JL)_*\subseteq I\subseteq B$. As shown before in the base step, there are $M^{\prime},M^{\prime\prime},N^{\prime},N^{\prime\prime}\in\Omega$ such that $J=(M^{\prime}M^{\prime\prime})_*$, $L=(N^{\prime}N^{\prime\prime})_*$ and $B=(M^{\prime}N^{\prime})_*$. We have that $(M^{\prime\prime}N^{\prime\prime}B)_*=(M^{\prime}M^{\prime\prime}N^{\prime}N^{\prime\prime})_*=(JL)_*\subseteq I=(AB)_*$, and hence $(M^{\prime\prime}N^{\prime\prime})_*\subseteq A$. It follows by the induction hypothesis that there are $C^{\prime},C^{\prime\prime},D^{\prime},D^{\prime\prime}\in\Omega$ such that $M^{\prime\prime}=(C^{\prime}C^{\prime\prime})_*$, $N^{\prime\prime}=(D^{\prime}D^{\prime\prime})_*$ and $A=(C^{\prime}D^{\prime})_*$. Set $J^{\prime}=(M^{\prime}C^{\prime})_*$, $J^{\prime\prime}=C^{\prime\prime}$, $L^{\prime}=(N^{\prime}D^{\prime})_*$ and $L^{\prime\prime}=D^{\prime\prime}$. Then $J^{\prime},J^{\prime\prime},L^{\prime},L^{\prime\prime}\in\Omega$, $J=(M^{\prime}M^{\prime\prime})_*=(J^{\prime}J^{\prime\prime})_*$, $L=(N^{\prime}N^{\prime\prime})_*=(L^{\prime}L^{\prime\prime})_*$ and $I=(AB)_*=(C^{\prime}D^{\prime}M^{\prime}N^{\prime})_*=(J^{\prime}L^{\prime})_*$.
\end{proof}

\section{$*$-Valuation ideal factorization domains}

A {\em $\pi$-domain} is an integral domain whose nonzero principal ideals can be written as a finite product of prime ideals \cite{a78}. Hence, each nonzero principal ideal of a $\pi$-domain can be written as a finite product of valuation ideals, because a prime ideal is a valuation ideal. In this section, we study such type of integral domains in the more general setting of star operations. We begin this section with the definition of $*$-VIFDs for which we note that an ideal $I$ of an integral domain $D$ is a fractional ideal of $D$ with $I\subseteq D$, so $D$ is also a $*$-ideal of $D$.

\begin{definition}\label{Definition 3.1}
{\em Let $D$ be an integral domain and let $*$ be a star operation on $D$. Then $D$ is called a {\em $*$-valuation ideal factorization domain} ($*$-VIFD) if each nonzero principal ideal $I$ of $D$ can be written as a finite $*$-product of valuation ideals of $D$, i.e., there are some $n\in\mathbb{N}$ and valuation ideals $I_i$ of $D$ such that $I=(\prod_{i=1}^n I_i)_*$. We say that $D$ is a {\em VIFD} if $D$ is a $d$-VIFD.}
\end{definition}

Let $*_1$ and $*_2$ be two star operations of finite type on $D$ such that $*_1\leq *_2$. It is easy to see that $(I_{*_1})_{*_2}=(I_{*_2})_{*_1}=I_{*_2}$ for all $I\in F(D)$. Hence, by definition, a $*_1$-VIFD is a $*_2$-VIFD. In particular,

\smallskip

\begin{center}
VIFD $\Rightarrow$ $*$-VIFD $\Rightarrow$ $t$-VIFD
\end{center}

\smallskip
\noindent
for any star operation $*$ of finite type.

\begin{lemma}\label{Lemma 3.2}
Let $D$ be an integral domain, let $*$ be a star operation of finite type on $D$ and let $I$ be a $*$-invertible valuation ideal of $D$.
Then $I_*$ is a $*$-invertible valuation $*$-ideal of $D$.
\end{lemma}

\begin{proof}
Since $*$-${\rm Max}(D)=\widetilde{*}$-${\rm Max}(D)$, a $*$-invertible ideal of $D$ is $\widetilde{*}$-invertible, so $I$ is a $\widetilde{*}$-invertible ideal of $D$. There exists a valuation overring $V$ of $D$ such that $IV\cap D=I$. Set $W=V_{\widetilde{*}}$ and note that $W$ is a valuation overring of $D$. Since $\widetilde{*}$ is stable, we have that $I_{\widetilde{*}}\subseteq I_{\widetilde{*}}W\cap D\subseteq (I_{\widetilde{*}}W)_{\widetilde{*}}\cap D=(IV)_{\widetilde{*}}\cap D_{\widetilde{*}}=(IV\cap D)_{\widetilde{*}}=I_{\widetilde{*}}$, and hence $I_{\widetilde{*}}W\cap D=I_{\widetilde{*}}$. We infer that $I_{\widetilde{*}}$ is a $\widetilde{*}$-invertible valuation ${\widetilde{*}}$-ideal of $D$. Therefore, $I_{\widetilde{*}}$ is a $*$-invertible valuation $*$-ideal of $D$. Finally, note that $I_{\widetilde{*}}=(I_{\widetilde{*}})_*=I_*$.
\end{proof}

We next give a first elementary characterization of $*$-VIFDs. Note that these characterizations view the concepts of $*$-VIFDs from three different angles. First, we can replace finite $*$-products of valuation ideals by finite $*$-products of valuation $*$-ideals (and vice versa). Second, we can lift the existence of representations of arbitrary nonzero principal ideals as finite $*$-products of valuation ideals to arbitrary $*$-invertible $*$-ideals. Finally, we prove the interchangeability of the star operations $*$ and $\widetilde{*}$ in this characterization.

\begin{proposition}\label{Proposition 3.3}
Let $D$ be an integral domain and let $*$ be a star operation of finite type on $D$\textnormal{:} The following statements are equivalent.
\begin{enumerate}
\item $D$ is a $*$-VIFD.
\item $D$ is a $\widetilde{*}$-VIFD.
\item Each nonzero principal ideal of $D$ is a finite $*$-product of valuation $*$-ideals.
\item Each $*$-invertible $*$-ideal of $D$ is a finite $*$-product of valuation $*$-ideals.
\item Each $*$-invertible $*$-ideal of $D$ is a finite $*$-product of valuation ideals.
\end{enumerate}
\end{proposition}

\begin{proof}
(1) $\Rightarrow$ (3) This is an immediate consequence of Lemma~\ref{Lemma 3.2}.

(2) $\Rightarrow$ (1) This follows from the fact that $\widetilde{*}\leq *$.

(3) $\Rightarrow$ (2) This is an immediate consequence of the following observation: If $I$ and $J$ are $*$-invertible $*$-ideals of $D$, then $I$ and $J$ are $\widetilde{*}$-invertible $\widetilde{*}$-ideals of $D$ and $(IJ)_{\widetilde{*}}=((IJ)_{\widetilde{*}})_*=(IJ)_*$ (since $(IJ)_{\widetilde{*}}$ is a $\widetilde{*}$-invertible $\widetilde{*}$-ideal of $D$).

(3) $\Rightarrow$ (4) Let $I$ be a proper $*$-invertible $*$-ideal of $D$. Choose a nonzero $a\in I$. Clearly, $aD=(\prod_{j=1}^m J_j)_*$ with $m\in\mathbb{N}$ and proper $*$-invertible valuation $*$-ideals $J_j$. Observe that $J_j$ is a $t$-invertible $t$-ideal for each $j\in [1,m]$. Consequently, for each $j\in [1,m]$, $J_j$ is contained in a unique maximal $t$-ideal by Corollary~\ref{Corollary 2.4}. We infer that $aD$ is contained in only finitely many maximal $t$-ideals of $D$, and hence $I$ is contained in only finitely many maximal $t$-ideals of $D$. Since $I$ is a $t$-invertible $t$-ideal of $D$, we have that $I_M$ is a principal ideal of $D_M$ for each $M\in t$-${\rm Max}(D)$. Let $N\in t$-${\rm Max}(D)$ be such that $I\subseteq N$. Then $I_N=bD_N$ for some $b\in I$. Note that $bD=(\prod_{i=1}^n I_i)_*$ for $n\in\mathbb{N}$ and $*$-invertible valuation $*$-ideals $I_i$ of $D$. Without restriction we can assume that there is some $r\in [1,n]$ such that $I_i\subseteq N$ for each $i\in [1,r]$ and $I_i\nsubseteq N$ for each $i\in [r+1,n]$. Observe that $I_N=((\prod_{i=1}^r I_i)_*)_N$. Let $\ell\in\mathbb{N}$ and let $(N_i)_{i=1}^{\ell}$ be the distinct maximal $t$-ideals of $D$ that contain $I$. Then for each $i\in [1,\ell]$ there is some $m_i\in\mathbb{N}$ and some finite $*$-product $(\prod_{j=1}^{m_i} J_{i,j})_*$ of $*$-invertible valuation $*$-ideals of $D$ such that $J_{i,j}\subseteq N_i$ for each $j\in [1,m_i]$ and such that $I_{N_i}=((\prod_{j=1}^{m_i} J_{i,j})_*)_{N_i}$. Note that for each $N\in t$-${\rm Max}(D)$, we have that $I_N=((\prod_{i=1}^{\ell}\prod_{j=1}^{m_i} J_{i,j})_*)_N$ (since every proper $*$-invertible valuation $*$-ideal of $D$ is contained in a unique maximal $t$-ideal of $D$). Since $I$ and $(\prod_{i=1}^{\ell}\prod_{j=1}^{m_i} J_{i,j})_*$ are $*$-invertible $*$-ideals of $D$ (and hence $t$-ideals of $D$), this implies that $I=(\prod_{i=1}^{\ell}\prod_{j=1}^{m_i} J_{i,j})_*$ is a finite $*$-product of valuation $*$-ideals of $D$.

(4) $\Rightarrow$ (5) $\Rightarrow$ (1) This is obvious.
\end{proof}

\begin{corollary}\label{Corollary 3.4}
Let $D$ be an integral domain and let $*_1$ and $*_2$ be star operations of finite type on $D$ such that $\widetilde{*_1}=\widetilde{*_2}$. Then $D$ is a $*_1$-VIFD if and only if $D$ is a $*_2$-VIFD.
\end{corollary}

\begin{proof}
It is an immediate consequence of Proposition~\ref{Proposition 3.3} that $D$ is a $*_1$-VIFD if and only if $D$ is a $\widetilde{*_1}$-VIFD if and only if $D$ is a $\widetilde{*_2}$-VIFD if and only if $D$ is a $*_2$-VIFD.
\end{proof}

\begin{corollary}\label{Corollary 3.5}
Let $D$ be an integral domain. The following statements are equivalent\textnormal{:}
\begin{enumerate}
\item $D$ is a VFD.
\item For each star operation $*$ of finite type on $D$, ${\rm Cl}_*(D)=\{0\}$ and every $*$-invertible $*$-ideal of $D$ is a finite $*$-product of valuation ideals.
\item For each star operation $*$ of finite type on $D$, $D$ is a $*$-VIFD and ${\rm Cl}_*(D)=\{0\}$.
\item There exists a star operation $*$ of finite type on $D$ such that ${\rm Cl}_*(D)=\{0\}$ and every $*$-invertible $*$-ideal of $D$ is a finite $*$-product of valuation ideals.
\item There exists a star operation $*$ of finite type on $D$ such that $D$ is a $*$-VIFD and ${\rm Cl}_*(D)=\{0\}$.
\end{enumerate}
\end{corollary}

\begin{proof}
(1) $\Rightarrow$ (2) Let $*$ be a star operation of finite type on $D$. Recall from \cite[Corollary 2.3(1)]{cr20} that ${\rm Cl}_t(D)=\{0\}$, so ${\rm Cl}_*(D)=\{0\}$. Now let $J$ be a proper $*$-invertible $*$-ideal of $D$. Then $J=aD$ for some nonzero nonunit $a\in D$. Since $D$ is a VFD, $J$ is a finite product of principal valuation ideals of $D$. Thus, $J$ is a finite $*$-product of principal valuation ideals of $D$.

(2) $\Rightarrow$ (3), (4) $\Rightarrow$ (5) These follow directly from the fact that a nonzero principal ideal is a $*$-invertible $*$-ideal.

(2) $\Rightarrow$ (4), (3) $\Rightarrow$ (5) Obvious, since $d$ is a star operation of finite type on $D$.

(5) $\Rightarrow$ (1) Let $a\in D$ be a nonzero nonunit. Then $aD=(\prod_{i=1}^n I_i)_*$ for some $*$-invertible valuation $*$-ideals $I_i$ of $D$ by Proposition~\ref{Proposition 3.3}. Clearly, each $I_i$ is principal, and hence $aD$ is a finite $*$-product of principal valuation ideals of $D$. Consequently, $aD$ is a finite product of principal valuation ideals of $D$. Thus, $a$ is a finite product of valuation elements of $D$.
\end{proof}

\begin{corollary}\label{Corollary 3.6}
Let $D$ be an integral domain and let $*_1$ and $*_2$ be star operations of finite type on $D$ such that $*_1\leq *_2$. Then $D$ is a VFD if and only if $D$ is a $*_1$-VIFD and ${\rm Cl}_{*_2}(D)=\{0\}$. In particular, $D$ is a VFD if and only if $D$ is a VIFD and ${\rm Pic}(D)=\{0\}$.
\end{corollary}

\begin{proof}
If $D$ is a VFD, then it is an immediate consequence of Corollary~\ref{Corollary 3.5} that $D$ is a $*_1$-VIFD and ${\rm Cl}_{*_2}(D)=\{0\}$. Now let $D$ be a $*_1$-VIFD and ${\rm Cl}_{*_2}(D)=\{0\}$. Then ${\rm Cl}_{*_1}(D)=\{0\}$, and hence $D$ is a VFD by Corollary~\ref{Corollary 3.5}. The additional statement is clear.
\end{proof}

A $\pi$-domain is a VIFD, because a prime ideal is a valuation ideal. We next study the relationship between a VFD and a VIFD, which is an analog of the fact that a UFD is a $\pi$-domain with trivial Picard group.

\begin{corollary}\label{Corollary 3.7}
Let $D$ be a Krull domain\textnormal{:}
\begin{enumerate}
\item $D$ is a $t$-VIFD.
\item $D$ is a VFD if and only if $D$ is a UFD.
\item $D$ is a VIFD if and only if $D$ is a $\pi$-domain.
\end{enumerate}
\end{corollary}

\begin{proof}
(1) This is clear.

(2) A Krull domain is a UFD if and only if its $t$-class group is trivial. Consequently, the result follows from (1) and Corollary~\ref{Corollary 3.5}.

(3) Observe that $t$-$\dim(D)\leq 1$. In general, a $\pi$-domain is a VIFD. Conversely, suppose that $D$ is a VIFD. It suffices to show that each height-one prime ideal of $D$ is invertible by \cite[Theorem 1]{a78}. Let $P$ be a height-one prime ideal. Choose $a\in P$ such that $aD_P=P_P$. Then $aD$ is a finite product of valuation ideals, so $P$ contains an invertible valuation ideal, say $Q$, containing $a$. By Proposition~\ref{Proposition 2.3}, $Q_P\cap D=Q$, and since $P_P=aD_P\subseteq Q_P$, it follows that $Q=P$. Thus, $P$ is invertible.
\end{proof}

Following \cite{ad14,dz11}, we say that $D$ is a {\em $*$-Schreier domain} if for all $*$-invertible $*$-ideals $I$, $J$ and $L$ of $D$ such that $(JL)_*\subseteq I$, there are some $*$-invertible $*$-ideals $J^{\prime}$ and $L^{\prime}$ of $D$ such that $J\subseteq J^{\prime}$, $L\subseteq L^{\prime}$ and $I=(J^{\prime}L^{\prime})_*$. Recall from \cite[Corollary 3.3]{ad14} that $D$ is a Schreier domain if and only if $D$ is an integrally closed $d$-Schreier domain with ${\rm Pic}(D)=\{0\}$, so there is a clear distinction between the concepts of Schreier and $d$-Schreier domains; see \cite[Proposition 2]{dz11}.

\begin{theorem}\label{Theorem 3.8}
Let $D$ be an integral domain and let $*$ be a star operation of finite type on $D$. Then $D$ is a $*$-VIFD if and only if $D$ is a $*$-Schreier domain and each nonzero prime $*$-ideal of $D$ contains a $*$-invertible valuation $(*$-$)$ideal of $D$.
\end{theorem}

\begin{proof}
($\Rightarrow$) First let $D$ be a $*$-VIFD. Then every $*$-invertible $*$-ideal of $D$ is a finite $*$-product of $*$-invertible valuation $*$-ideals of $D$ by Proposition~\ref{Proposition 3.3}. In particular, the set of $*$-invertible $*$-ideals is the set of finite $*$-products of $*$-invertible valuation $*$-ideals of $D$. We infer by Proposition~\ref{Proposition 2.12} that $D$ is a $*$-Schreier domain. Now let $P$ be a nonzero prime $*$-ideal of $D$. Then there is a nonzero nonunit $a\in D$ such that $aD\subseteq P$. Since $aD$ is a finite $*$-product of $*$-invertible valuation $*$-ideals of $D$, we have that $P$ contains a $*$-invertible valuation $*$-ideal of $D$.

($\Leftarrow$) Now let $D$ be a $*$-Schreier domain such that each nonzero prime $*$-ideal of $D$ contains a $*$-invertible valuation $*$-ideal of $D$. Let $\Omega$ be the set of all finite $*$-products of $*$-invertible valuation $*$-ideals of $D$. Assume that $D$ is not a $*$-VIFD. Then there exists a $*$-invertible $*$-ideal $I$ of $D$ such that $I\not\in\Omega$. Let $\Sigma=\{J\mid J$ is a $*$-ideal of $D$ such that $I\subseteq J$ and $A\nsubseteq J$ for each $A\in\Omega\}$. Assume that $I\not\in\Sigma$. Then there are $n\in\mathbb{N}$ and $*$-invertible valuation $*$-ideals $I_i$ of $D$ such that $(\prod_{i=1}^n I_i)_*\subseteq I$. Since $D$ is a $*$-Schreier domain, it follows (by induction) that there are some $*$-invertible $*$-ideals $J_i$ of $D$ such that $I=(\prod_{i=1}^n J_i)_*$ and $I_j\subseteq J_j$ for each $j\in [1,n]$. We infer by Proposition~\ref{Proposition 2.9}(3) that $J_j$ is a valuation ideal of $D$ for each $j\in [1,n]$. Therefore, $I\in\Omega$, a contradiction. Hence, $I\in\Sigma$. Then $\Sigma\not=\emptyset$, and since $*$ is of finite type and each element of $\Omega$ is $*$-finite, we have that $\Sigma$ is ordered inductively (under inclusion). Consequently, there is a maximal element $P\in\Sigma$ by Zorn's lemma. We show that $P$ is a prime $*$-ideal of $D$. Clearly, $P$ is a proper $*$-ideal of $D$. Assume that $P$ is not a prime ideal of $D$, then there are $a,b\in D$ such that $ab\in P$ and $a,b\not\in P$. We have that $P\subsetneq (P+aD)_*$ and $P\subsetneq (P+bD)_*$, and hence $(P+aD)_*\not\in\Sigma$ and $(P+bD)_*\not\in\Sigma$. Consequently, there are $A,B\in\Omega$ such that $A\subseteq (P+aD)_*$ and $B\subseteq (P+bD)_*$. This implies that $(AB)_*\subseteq (P^2+aP+bP+abD)_*\subseteq P$ and $(AB)_*\in\Omega$, a contradiction. Hence, $P$ is a prime $*$-ideal of $D$. Now since $I\subseteq P$, we have that $P$ is nonzero, and thus $P$ contains a $*$-invertible valuation $*$-ideal $J$ of $D$. Note that $J\in\Omega$, a contradiction.
\end{proof}

The next result is a valuation ideal analog of \cite[Proposition 1.7(4)]{cr20} that if $a\in D$ is a valuation element, then either $a$ is a unit of $D_S$ or $a$ is a valuation element of $D_S$ for any multiplicatively closed subset $S$ of $D$.

\begin{lemma}\label{Lemma 3.9}
Let $D$ be an integral domain, let $S$ be a multiplicatively closed subset of $D$ and let $I$ be a valuation ideal of $D$. Then $I_S$ is a valuation ideal of $D_S$.
\end{lemma}

\begin{proof}
Let $V$ be a valuation overring of $D$ such that $IV\cap D=I$. Then $I_S=(IV\cap D)_S=I_SV_S\cap D_S$ and $V_S$ is a valuation overring of $D_S$. Thus, $I_S$ is a valuation ideal of $D_S$.
\end{proof}

Let $*$ be a star operation of finite type on $D$ and let $*_S:F(D_S)\rightarrow F(D_S)$ be defined by $(I_S)_{*_S}=(I_*)_S$ for each $I\in F(D)$. Then $*_S$ is a star operation of finite type on $D_S$ \cite[4.4 Theorem]{h98}. If $P$ is a prime ideal of $D$ such that $S=D\setminus P$, then we write $*_P$ instead of $*_S$.

\begin{proposition}\label{Proposition 3.10}
Let $D$ be an integral domain, let $*$ be a star operation of finite type on $D$ and let $P$ be a prime $*$-ideal of $D$. If $D$ is a $*$-VIFD, then $D_P$ is a VFD.
\end{proposition}

\begin{proof}
Let $a\in D$ be a nonzero nonunit of $D_P$. Then $aD=(\prod_{i=1}^n I_i)_*$ for some $n\in\mathbb{N}$ and valuation $*$-ideals $I_i$ of $D$. If $j\in [1,n]$, then $I_j$ is $*$-invertible, and hence $(I_j)_P=a_jD_P$ for some $a_j\in I_j$. Therefore,

\begin{eqnarray*}
aD_P &=& \Big(\Big(\prod_{i=1}^n I_i\Big)_*\Big)_P=\Big(\Big(\prod_{i=1}^n I_i\Big)_P\Big)_{*_P}\\
&=& \Big(\prod_{i=1}^n (I_i)_P\Big)_{*_P}\\
&=& \Big(\prod_{i=1}^n (a_iD_P)\Big)_{*_P}\\
&=& \Big(\prod_{i=1}^n a_i\Big)D_P.
\end{eqnarray*}

\noindent
We infer by Lemma~\ref{Lemma 3.9} that $a_i$ is either a unit of $D_P$ or a valuation element of $D_P$ for each $i\in [1,n]$. Consequently, $D_P$ is a VFD.
\end{proof}

An integral domain is called a {\em Mori domain} if it satisfies the ACC on $t$-ideals. Moreover, we say that $D$ is {\em completely integrally closed} if for each $x\in K$ for which there is some nonzero $c\in D$ such that $cx^n\in D$ for all $n\in\mathbb{N}$, it follows that $x\in D$. Observe that $D$ is a {\em Krull domain} if and only if $D$ is a completely integrally closed Mori domain \cite[Theorem 2.3.11]{gh06}. The purpose of the next result is to generalize \cite[Corollary 2.4]{cr20} that characterizes when a VFD is a Mori domain.

\begin{corollary}\label{Corollary 3.11}
Let $D$ be a $t$-VIFD. The following statements are equivalent\textnormal{:}
\begin{enumerate}
\item $D$ is a Mori domain.
\item $D_M$ is a UFD for each $M\in t$-${\rm Max}(D)$.
\item $D$ is a Krull domain.
\end{enumerate}
\end{corollary}

\begin{proof}
(1) $\Rightarrow$ (2) Let $M\in t$-${\rm Max}(D)$. Then $D_M$ is a Mori domain \cite[Proposition 2.10.4.2]{gh06}. Moreover, $D_M$ is a VFD by Proposition~\ref{Proposition 3.10}. Thus, $D_M$ is a UFD \cite[Corollary 2.4]{cr20}.

(2) $\Rightarrow$ (3) Clearly, $D_M$ is a Krull domain for each $M\in t$-${\rm Max}(D)$. Furthermore, $D$ is of finite $t$-character by Corollary~\ref{Corollary 2.4}. Consequently, $D$ is a Krull domain.

(3) $\Rightarrow$ (1) This is obvious.
\end{proof}

Let $*$ be a star operation of finite type on $D$. Next, we are going to show that a $*$-VIFD is an integrally closed weakly Matlis domain. However, since $*\leq t$, a $*$-VIFD is a $t$-VIFD, so it suffices to show that a $t$-VIFD is an integrally closed weakly Matlis domain.

\begin{proposition}\label{Proposition 3.12}
Let $D$ be an integral domain such that every nonzero prime $t$-ideal of $D$ contains a $t$-invertible valuation ideal of $D$. Then $D$ is an integrally closed weakly Matlis domain.
\end{proposition}

\begin{proof}
It is an immediate consequence of Lemma~\ref{Lemma 3.2} that every nonzero prime $t$-ideal of $D$ contains a $t$-invertible valuation $t$-ideal. Let $\overline{D}$ be the integral closure of $D$ and let $\Omega$ be the set of $t$-invertible $t$-ideals $I$ of $D$ such that $I\overline{D}\cap D=I$. It follows from Proposition~\ref{Proposition 2.9} that $\Omega$ is a multiplicatively closed and divisor-closed subset of the monoid of $t$-invertible $t$-ideals of $D$. (The notions of multiplicatively closed and divisor-closed can be defined in analogy for monoids. For instance, see \cite{gh06}.) Assume that $\Omega$ is not the set of all $t$-invertible $t$-ideals of $D$. Then there exists a $t$-invertible $t$-ideal $J$ of $D$ such that $J\not\in\Omega$. Since $\Omega$ is divisor-closed, we infer that $L\nsubseteq J$ for each $L\in\Omega$. Let $\Sigma=\{A\mid A$ is a $t$-ideal of $D$ such that $J\subseteq A$ and $L\nsubseteq A$ for each $L\in\Omega\}$. It is clear that $J\in\Sigma$. Observe that $\Sigma$ is ordered inductively under inclusion (since each element of $\Omega$ is $t$-finite). Consequently, $\Sigma$ has a maximal element $Q$ by Zorn's lemma. It is straightforward to show that $Q$ is a nonzero prime $t$-ideal of $D$; e.g., as in the proof of Theorem~\ref{Theorem 3.8}. Hence, $Q$ contains a $t$-invertible valuation $t$-ideal $B$ of $D$. On the other hand, we have that $B\in\Omega$, a contradiction. We infer that $\Omega$ is the set of all $t$-invertible $t$-ideals of $D$. Next we show that $D$ is integrally closed. It remains to prove that $\overline{D}\subseteq D$. Let $x\in\overline{D}$. Then $x=\frac{a}{b}$ for some $a\in D$ and some nonzero $b\in D$. Observe that $bD\in\Omega$. It follows that $a=bx\in b\overline{D}\cap D=(bD)\overline{D}\cap D=bD$. Thus, $x\in D$.

Finally, we show that $D$ is weakly Matlis. It follows from Proposition~\ref{Proposition 2.3}(1) that for each nonzero prime $t$-ideal $P$ of $D$, there exists a $t$-invertible $t$-ideal $I$ of $D$ and an $M\in t$-${\rm Max}(D)$ such that $I\subseteq P$ and $I_M\cap D=I$. We infer by \cite[Theorem 4.3]{amz92} that $D$ is weakly Matlis.
\end{proof}

It is an immediate consequence of Proposition~\ref{Proposition 3.12} that if every nonzero prime $*$-ideal of $D$ contains a $*$-invertible valuation ideal, then $D$ is an integrally closed weakly Matlis domain. Nevertheless, it follows from Example~\ref{Example 2.5} that $D$ need not be $*$-${\rm h}$-local.

\begin{corollary}\label{Corollary 3.13}
Let $D$ be an integral domain and let $*$ be a star operation of finite type on $D$. If $D$ is a $*$-VIFD, then $D$ is an integrally closed weakly Matlis domain.
\end{corollary}

\begin{proof}
This is an immediate consequence of Proposition~\ref{Proposition 3.12}.
\end{proof}

We say that $D$ is a {\em $*$-treed domain} if the set of prime $*$-ideals of $D$ is treed under inclusion. Hence, $D$ is $*$-treed if and only if ${\rm Spec}(D_M)$ is linearly ordered under inclusion for all maximal $*$-ideals $M$ of $D$. The class of $*$-treed domains includes P$*$MDs, integral domains of $*$-dimension one, and treed domains. Moreover, $D$ is said to be a {\em ring of Krull type} if $D$ is a P$v$MD of finite $t$-character and $D$ is called an {\em independent ring of Krull type} if $D$ is a weakly Matlis P$v$MD. Next we recall the connections between P$*$MDs and P$v$MDs, as well as the connections between $*$-treed domains and $t$-treed domains. Observe that $\mathbb{R}[X,Y]$ is a P$v$MD (and thus $t$-treed), but it is not treed (and hence it is not a Pr\"ufer domain).

\begin{remark}\label{Remark 3.14}
{\em Let $D$ be an integral domain and let $*$ be a star operation of finite type on $D$\textnormal{:}
\begin{enumerate}
\item $D$ is a P$*$MD if and only if $D$ is a P$v$MD and $*=t$.
\item $D$ is $*$-treed if and only if $D$ is $t$-treed and $\widetilde{*}=w$.
\end{enumerate}}
\end{remark}

\begin{proof}
(1) This follows from \cite[Proposition 3.15]{fjs03}.

(2) ($\Rightarrow$) Let $D$ be $*$-treed. It is sufficient to show that $*$-${\rm Spec}(D)=t$-${\rm Spec}(D)$. (Then obviously $D$ is $t$-treed and $*$-${\rm Max}(D)=t$-${\rm Max}(D)$, and thus $\widetilde{*}=w$ by \cite[Lemma 2.1(2)]{or20}.) Clearly, $t$-${\rm Spec}(D)\subseteq *$-${\rm Spec}(D)$. Now let $P\in *$-${\rm Spec}(D)$. There is some $Q\in t$-${\rm Spec}(D)$ such that $Q\subseteq P$ and such that for each $Q^{\prime}\in t$-${\rm Spec}(D)$ with $Q\subseteq Q^{\prime}\subseteq P$, it follows that $Q^{\prime}=Q$. Assume that $Q\subsetneq P$. There are some $x\in P\setminus Q$ and $Q^{\prime}\in\mathcal{P}(xD)$ such that $Q^{\prime}\subseteq P$. Observe that $Q^{\prime}\in t$-${\rm Spec}(D)\subseteq *$-${\rm Spec}(D)$. Since $D$ is $*$-treed and $Q^{\prime}\nsubseteq Q$, we have that $Q\subseteq Q^{\prime}$. Consequently, $Q^{\prime}=Q$, a contradiction. We infer that $P=Q\in t$-${\rm Spec}(D)$.

($\Leftarrow$) Let $D$ be $t$-treed and let $\widetilde{*}=w$. It follows from \cite[Lemma 2.1(2)]{or20} that $*$-${\rm Max}(D)=\widetilde{*}$-${\rm Max}(D)=w$-${\rm Max}(D)=t$-${\rm Max}(D)$. Since $D$ is $t$-treed, we obtain that ${\rm Spec}(D_M)$ is linearly ordered for each $M\in t$-${\rm Max}(D)$, and hence ${\rm Spec}(D_M)$ is linearly ordered for each $M\in *$-${\rm Max}(D)$. This implies that $D$ is $*$-treed.
\end{proof}

\begin{theorem}\label{Theorem 3.15}
Let $D$ be an integral domain and let $*$ be a star operation of finite type on $D$. The following statements are equivalent\textnormal{:}
\begin{enumerate}
\item $D$ is a $*$-${\rm h}$-local P$*$MD.
\item Every nonzero $(*$-finite$)$ $*$-ideal of $D$ can be written as a finite $*$-product of $*$-comaximal valuation ideals.
\item Every nonzero principal ideal of $D$ can be written as a finite $*$-product of $*$-comaximal valuation ideals.
\item $D$ is $*$-treed and $D$ is a $*$-VIFD.
\item $D$ is $*$-treed and every nonzero prime ideal of $D$ contains a $*$-invertible valuation ideal.
\item $D$ is $*$-treed and every nonzero $*$-ideal of $D$ is a finite $*$-product of valuation ideals.
\end{enumerate}
\end{theorem}

\begin{proof}
(1) $\Rightarrow$ (2), (6) Since $D$ is a P$*$MD, it is clear that $D$ is $*$-treed. It remains to show that every nonzero proper $*$-ideal of $D$ is a finite $*$-product of $*$-comaximal valuation ideals. Let $I$ be a nonzero proper $*$-ideal of $D$. Then $\bigcap_{M\in *\textnormal{-}{\rm Max}(D)}I_M=I$, and since $D$ is of finite $*$-character, there are only finitely many maximal $*$-ideals, say, $(M_i)_{i=1}^k$, such that $I=(\bigcap_{i=1}^k I_{M_i})\cap D$. Let $I_i=I_{M_i}\cap D$ for $i\in [1,k]$. Then, since $D_{M_i}$ is a valuation domain, each $I_i$ is a valuation $*$-ideal. Note that if $i,j\in [1,k]$ are distinct, then $(D_{M_i})_{M_j}$ is the quotient field of $D$. Hence, $(I_i+I_j)_*=D$ for all distinct $i,j\in [1,k]$, and thus $\bigcap_{i=1}^k I_i=(\prod_{i=1}^k I_i)_*$. Therefore, $I=(\prod_{i=1}^k I_i)_*$ is a finite $*$-product of the $*$-comaximal valuation ideals $I_i$.

(2) $\Rightarrow$ (3) This is clear.

(3) $\Rightarrow$ (4) Let $M$ be a maximal $*$-ideal of $D$ and let $a\in M$ be nonzero. Then, by assumption, $aD=(\prod_{i=1}^n Q_i)_*$ for some $n\in\mathbb{N}$ and proper $*$-comaximal valuation ideals $Q_i$ of $D$. Note that $Q_i$ and $Q_j$ are $*$-comaximal for each distinct $i,j\in [1,n]$, so $M$ contains exactly one of the $Q_i$'s, say $Q_1$ for convenience. Consequently,

\[
aD_M=\Big(\Big(\prod_{i=1}^n Q_i\Big)_*\Big)_M=\Big(\prod_{i=1}^n (Q_i)_M\Big)_{*_M}=(Q_1)_M
\]
\noindent
and $(Q_1)_M$ is a valuation ideal of $D_M$ by Lemma~\ref{Lemma 3.9}, which means that $a$ is a valuation element of $D_M$. Thus, every nonzero nonunit of $D_M$ is a valuation element, and hence $D_M$ is a valuation domain \cite[Corollary 2.4]{go65}. Therefore, $D$ is $*$-treed.

(4) $\Rightarrow$ (5) Let $P$ be a nonzero prime ideal of $D$, and choose a nonzero $a\in D$. Then $aD=(\prod_{i=1}^n Q_i)_*$ for some $n\in\mathbb{N}$ and valuation ideals $Q_i$ of $D$. Clearly, for each $i\in [1,n]$, $Q_i$ is $*$-invertible and $P$ contains at least one of the $Q_i$'s.

(5) $\Rightarrow$ (1) It follows from Proposition~\ref{Proposition 3.12} that $D$ is a weakly Matlis domain. It suffices to show that $D$ is a P$*$MD; equivalently, $D_M$ is a valuation domain for each $M\in *$-${\rm Max}(D)$. Now let $M$ be a maximal $*$-ideal of $D$. Then ${\rm Spec}(D_M)$ is linearly ordered under inclusion because $*$-${\rm Spec}(D)$ is treed, and each prime ideal of $D_M$ contains a valuation element by assumption and Lemma~\ref{Lemma 3.9}. Now if $b\in D_M$ is a nonzero nonunit, then $\sqrt{bD_M}$ is a prime ideal, so there is a valuation element $c\in\sqrt{bD_M}$. Hence, $\sqrt{cD_M}\subseteq\sqrt{bD_M}$, and thus $b$ is a valuation element of $D_M$ \cite[Proposition 1.1(3)]{cr20}. Thus, $D_M$ is a valuation domain \cite[Corollary 1.4]{cr20}.

(6) $\Rightarrow$ (4) This is obvious.
\end{proof}

\begin{corollary}\label{Corollary 3.16}
Let $D$ be a $t$-treed domain. Then $D$ is a VFD if and only if ${\rm Cl}_t(D)=\{0\}$ and every nonzero prime ideal of $D$ contains a valuation element.
\end{corollary}

\begin{proof}
($\Rightarrow$) It is clear that every nonzero prime ideal of $D$ contains a valuation element. Thus, the result follows because a VFD has a trivial $t$-class group \cite[Corollary 2.3(1)]{cr20}.

($\Leftarrow$) If $a\in D$ is a valuation element, then $aD$ is a $t$-invertible valuation ideal. Hence, $D$ is an independent ring of Krull type by Theorem~\ref{Theorem 3.15}. Therefore, ${\rm Cl}_t(D)=\{0\}$ implies that $D$ is a weakly Matlis GCD-domain, so $D$ is a VFD \cite[Theorem 3.4]{cr20}.
\end{proof}

\begin{corollary}\label{Corollary 3.17}
Let $D$ be an integral domain and let $*$ be a star operation of finite type on $D$ such that $*$-$\dim(D)=1$. The following statements are equivalent\textnormal{:}
\begin{enumerate}
\item $D$ is a P$*$MD of finite $*$-character.
\item $D$ is a $*$-VIFD.
\item Every nonzero prime ideal of $D$ contains a $*$-invertible valuation ideal.
\item Each nonzero $*$-ideal of $D$ is a finite $*$-product of valuation ideals.
\end{enumerate}
\end{corollary}

\begin{proof}
Since $*$-$\dim(D)=1$, we have that $D$ is $*$-treed. Moreover, $D$ is $*$-${\rm h}$-local if and only if $D$ is of finite $*$-character (since every nonzero prime $*$-ideal of $D$ is a maximal $*$-ideal). Now the equivalence is an immediate consequence of Theorem~\ref{Theorem 3.15}.
\end{proof}

Next we characterize when a $*$-VIFD with $*$-$\dim(D)=1$ is atomic, which is a variant of \cite[Theorem 4.3]{ch18}, because $D$ is a $*$-VIFD if $D$ is a P$*$MD of finite $*$-character and $*$-$\dim(D)=1$ by Corollary~\ref{Corollary 3.17}. We first need a definition. Assume that $D$ is a P$*$MD of finite $*$-character and $*$-$\dim(D)=1$. Let $M\in *$-${\rm Max}(D)$. Then $D_M$ is a one-dimensional valuation domain, and hence there exists a valuation ${\rm v}_M:K\setminus\{0\}\rightarrow\mathbb{R}$ (where $\mathbb{R}$ is the additive group of real numbers) such that $D_M\setminus\{0\}=\{x\in K\setminus\{0\}\mid {\rm v}_M(x)\geq 0\}$. For each $x\in K\setminus\{0\}$, let $||x||=\sum_{M\in *\textnormal{-}{\rm Max}(D)} {\rm v}_M(x)$ (this is well-defined, since $D$ is of finite $*$-character). Observe that for each nonzero $a,b\in D$, $a\mid_D b$ if and only if ${\rm v}_N(a)\leq {\rm v}_N(b)$ for each $N\in *$-${\rm Max}(D)$. Moreover, $||xy||=||x||+||y||$ for all $x,y\in K\setminus\{0\}$. Let $\mathbb{R}_{>0}$ denote the set of positive real numbers.

We say that $D$ {\em satisfies the ACCP} if $D$ satisfies the ACC on principal ideals of $D$. Moreover, $D$ is called a {\em BF-domain} if $D$ is atomic and for each nonzero nonunit $b\in D$, there is some $n\in\mathbb{N}$ such that $b$ is not a product of more than $n$ atoms of $D$. The next result can be proved along the same lines as \cite[Theorem 4.3]{ch18}, but for the sake of completeness, we include a proof.

\begin{proposition}\label{Proposition 3.18}
Let $D$ be an integral domain and let $*$ be a star operation of finite type on $D$ such that $D$ is a P$*$MD, $D$ is of finite $*$-character and $*$-$\dim(D)=1$. The following statements are equivalent\textnormal{:}
\begin{enumerate}
\item For each nonzero nonunit $a\in D$, there exists some $r\in\mathbb{R}_{>0}$ such that $||b||\geq r$ for each nonzero nonunit $b\in D$ such that $b\mid_D a$.
\item $D$ is a BF-domain.
\item $D$ satisfies the ACCP.
\item $D$ is atomic.
\end{enumerate}
\end{proposition}

\begin{proof}
(1) $\Rightarrow$ (2) Let $a\in D$ be a nonzero nonunit and let $t=\inf\{||b||\mid b\in D$ is a nonzero nonunit such that $b\mid_D a\}$. Note that $t>0$. Set $k=\lfloor\frac{||a||}{t}\rfloor$. It is sufficient to show that if $\ell\in\mathbb{N}$ is such that $a$ is a product of $\ell$ nonunits of $D$, then $\ell\leq k$. Let $\ell\in\mathbb{N}$ be such that $a$ is a product of $\ell$ nonunits of $D$. Then $a=\prod_{i=1}^{\ell} a_i$ for some nonunits $a_i\in D$. Observe that $t\ell\leq\sum_{i=1}^{\ell} ||a_i||=||a||$, and thus $\ell\leq k$.

(2) $\Rightarrow$ (3) $\Rightarrow$ (4) This follows from \cite[Propositions 1.1.4 and 1.3.2]{gh06}.

(4) $\Rightarrow$ (1) Let $a\in D$ be a nonzero nonunit. Note that $\mathcal{P}(aD)$ is the set of maximal $*$-ideals of $D$ that contain $aD$ (since $*$-$\dim(D)=1$). Set $\Omega=\{\mathcal{P}(uD)\mid u\in D$ is an atom such that $\mathcal{P}(uD)\subseteq\mathcal{P}(aD)\}$. Then $\Omega$ is finite (since $D$ is of finite $*$-character) and $\Omega\not=\emptyset$ (since $D$ is atomic). Consequently, there exists a finite nonempty set $\mathcal{A}$ of atoms of $D$ such that $\Omega=\{\mathcal{P}(uD)\mid u\in\mathcal{A}\}$. Note that $\{{\rm v}_P(u)\mid u\in\mathcal{A},P\in\mathcal{P}(uD)\}$ is finite and nonempty, since $\mathcal{A}$ is finite and nonempty and $D$ is of finite $*$-character. Set $r=\min\{{\rm v}_P(u)\mid u\in\mathcal{A},P\in\mathcal{P}(uD)\}$. Then $r\in\mathbb{R}_{>0}$.

Assume that there exists a nonzero nonunit $b\in D$ with $||b||<r$ and $b\mid_D a$. Since $D$ is atomic, there is an atom $v\in D$ with $v\mid_D b$. We have that $\mathcal{P}(vD)\in\Omega$ (since $v\mid_D a$), and hence $\mathcal{P}(vD)=\mathcal{P}(uD)$ for some $u\in\mathcal{A}$. Let $M\in *$-${\rm Max}(D)$. If $v\not\in M$, then $u\not\in M$, and thus ${\rm v}_M(v)=0={\rm v}_M(u)$. If $v\in M$, then $u\in M$, and hence ${\rm v}_M(v)\leq ||v||\leq ||b||\leq r\leq {\rm v}_M(u)$. In any case, we have that ${\rm v}_N(v)\leq {\rm v}_N(u)$ for each $N\in *$-${\rm Max}(D)$. Consequently, $v\mid_D u$. This implies that $u=v\varepsilon$ for some $\varepsilon\in D^{\times}$ (since $u$ and $v$ are atoms of $D$). There exists some $P\in\mathcal{P}(vD)=\mathcal{P}(uD)$. Note that ${\rm v}_P(u)={\rm v}_P(v)\leq ||v||\leq ||b||<r\leq {\rm v}_P(u)$, a contradiction.
\end{proof}

\smallskip
\section{VIFDs and $t$-VIFDs}

We begin this section with easy examples of VIFDs, which are Dedekind domains. Dedekind domains are integral domains whose nonzero ideals can be written as a finite product of prime ideals. Then Dedekind domains are $\pi$-domains, and $D$ is a Dedekind domain if and only if $D$ is a $\pi$-domain of Krull dimension at most one.

\begin{example}\label{Example 4.1}
{\em Let $D$ be a Dedekind domain. Then $D$ is a VIFD because each prime ideal is a valuation ideal. Moreover, note that a Dedekind domain is a PID if and only if its ideal class group is trivial. Note also that if $D$ is a VFD, then ${\rm Cl}_t(D)=\{0\}$ \cite[Corollary 2.3(1)]{cr20}. Hence, $D$ is a VFD if and only if $D$ is a PID.}
\end{example}

As in \cite{o00}, we say that $D$ is a {\em ZPUI domain} if every nonzero proper ideal $I$ of $D$ can be written as $I=J\prod_{i=1}^n P_i$, where $J$ is an invertible ideal of $D$, $n\in\mathbb{N}$ and the $(P_i)_{i=1}^n$ are prime ideals of $D$. It is known that $D$ is a ZPUI domain if and only if $D$ is a strongly discrete ${\rm h}$-local Pr\"ufer domain \cite[Theorem 2.3]{o00}. As a $w$-operation analog, we say that $D$ is a {\em $w$-ZPUI domain} if every nonzero proper $w$-ideal of $D$ can be written as $I=(J\prod_{i=1}^n P_i)_{w}$ for some $w$-invertible ideal $J$ of $D$, $n\in\mathbb{N}$ and $(P_i)_{i=1}^n$ pairwise $w$-comaximal prime $w$-ideals of $D$ \cite[Definition 3.1]{cc23}. It is known that $D$ is a $w$-ZPUI domain if and only if $D$ is a strongly discrete independent ring of Krull type \cite[Theorem 3.5]{cc23}.

Following \cite{egz08}, we say that $D$ is a {\em unique representation domain} (URD) if each $t$-invertible $t$-ideal of $D$ can be uniquely expressed as a finite $t$-product of pairwise $t$-comaximal $t$-ideals with prime radical. Then $D$ is a URD if and only if each nonzero principal ideal of $D$ can be written as a finite $t$-product of pairwise $t$-comaximal $t$-ideals with prime radical, if and only if $D$ is $t$-treed and each nonzero principal of $D$ has only finitely many minimal prime ideals \cite[Corollary 2.12]{egz08}. Hence, we have the following corollary, while a URD need not be a $w$-VIFD; see, for example, \cite[Corollary 2.17]{egz08} and Corollary~\ref{Corollary 4.6}.

\begin{proposition}\label{Proposition 4.2}
Let $D$ be an integral domain\textnormal{:}
\begin{enumerate}
\item If $D$ is a ZPUI domain, then $D$ is a VIFD.
\item If $D$ is a $w$-ZPUI domain, then $D$ is a $w$-VIFD.
\item If $D$ is a $w$-VIFD that is $t$-treed, then $D$ is a URD.
\end{enumerate}
\end{proposition}

\begin{proof}
(1) Let $D$ be a ZPUI domain. Then $D$ is an ${\rm h}$-local Pr\"ufer domain, and thus the result follows by Theorem~\ref{Theorem 3.15}.

(2) A $w$-ZPUI domain is an independent ring of Krull type. Thus, if $D$ is a $w$-ZPUI domain, then $D$ is a $w$-VIFD by Theorem~\ref{Theorem 3.15}.

(3) This follows from Proposition~\ref{Proposition 2.3}(2), Lemma~\ref{Lemma 3.2} and Theorem~\ref{Theorem 3.15}.
\end{proof}

The next result shows that the localization of a $w$-VIFD is also a $w$-VIFD as in the case of VIFD that every localization of a VIFD is a VIFD by Lemma~\ref{Lemma 3.9}.

\begin{proposition}\label{Proposition 4.3}
Let $D$ be a $w$-VIFD and let $S$ be a multiplicatively closed subset of $D$\textnormal{:}
\begin{enumerate}
\item $D_S$ is a $w$-VIFD.
\item If ${\rm Cl}_t(D_S)=\{0\}$, then $D_S$ is a VFD.
\item If $S=D\setminus\bigcup_{i=1}^n M_i$ for some $n\in\mathbb{N}$ and maximal $t$-ideals $M_i$ of $D$, then $D_S$ is a VFD.
\end{enumerate}
\end{proposition}

\begin{proof}
(1) Let $A$ be a nonzero principal ideal of $D_S$. Then $A=aD_S$ for some nonzero nonunit $a$ of $D$. Hence, by assumption, $aD=(\prod_{i=1}^n I_i)_w$ for some valuation ideals $I_i$ of $D$, and since $\prod_{i=1}^n I_i$ is $t$-invertible, we have that

\[
A=\Big(\Big(\prod_{i=1}^n I_i\Big)_w\Big)_S=\Big(\Big(\prod_{i=1}^n I_i\Big)_S\Big)_{w_S}=\Big(\prod_{i=1}^n (I_i)_S\Big)_{w_S}=\Big(\prod_{i=1}^n (I_i)_S\Big)_t.
\]

\noindent
The last equality holds, since $(\prod_{i=1}^n (I_i)_S)_{w_S}$ is a $w_S$-invertible $w_S$-ideal of $D_S$, and thus it is a $t$-ideal of $D_S$. Note that $(I_i)_S$ is a valuation ideal by Lemma~\ref{Lemma 3.9}, and hence $D_S$ is a $w$-VIFD by Proposition~\ref{Proposition 3.3}.

(2) This follows from (1) and Corollary~\ref{Corollary 3.5}.

(3) If $I$ is a $t$-invertible ideal of $D$, then $II^{-1}\nsubseteq\bigcup_{i=1}^n M_i$, and hence $I_S$ is invertible. Thus, $I_S$ is principal because ${\rm Pic}(D_S)=\{0\}$. Thus, by the proof of (1) above, every nonzero principal ideal of $D_S$ can be written as a finite ($t$-)product of principal valuation ideals, which implies that $D_S$ is a VFD.
\end{proof}

Let $D[X]$ be the polynomial ring over $D$. A nonzero prime ideal $Q$ of $D[X]$ is called an {\em upper to zero} in $D[X]$ if $Q\cap D=(0)$. Following \cite{hz89}, we say that $D$ is a {\em UMT-domain} if each upper to zero in $D[X]$ is a maximal $t$-ideal of $D[X]$. It is known that $D$ is a P$v$MD if and only if $D$ is an integrally closed UMT-domain \cite[Proposition 3.2]{hz89}.

\begin{lemma}\label{Lemma 4.4}
Let $D$ be a UMT-domain, let $D[X]$ be the polynomial ring over $D$ and let $N_v=\{f\in D[X]\mid f\neq 0$ and $c(f)_v=D\}$. Then $D$ is a $t$-treed domain if and only if $D[X]_{N_v}$ is treed.
\end{lemma}

\begin{proof}
The result follows directly from the fact that $D$ is a UMT-domain if and only if each prime ideal of $D[X]_{N_v}$ is extended from $D$ \cite[Theorem 3.1]{hz89}.
\end{proof}

It is easy to see that if $D$ is a Krull domain, then $D$ is a $t$-VIFD, $t$-${\rm Spec}(D)$ is treed, and $D$ is an independent ring of Krull type. Now we characterize when $D$ is a $t$-VIFD under the assumption that $D$ is $t$-treed.

\begin{theorem}\label{Theorem 4.5}
Let $D$ be a $t$-treed domain, let $D[X]$ be the polynomial ring over $D$ and let $N_v=\{f\in D[X]\mid f\neq 0$ and $c(f)_v=D\}$. The following statements are equivalent\textnormal{:}
\begin{enumerate}
\item Every nonzero $t$-ideal of $D$ is a finite $t$-product of valuation ideals.
\item $D$ is a $t$-VIFD.
\item $D$ is an independent ring of Krull type.
\item $D[X]$ is an independent ring of Krull type.
\item $D[X]$ is a $t$-VIFD and $D$ is a UMT-domain.
\item $D[X]_{N_v}$ is a VFD and $D$ is a UMT-domain.
\item $D[X]_{N_v}$ is an ${\rm h}$-local Pr\"ufer domain.
\item Every nonzero prime ideal of $D$ contains a $t$-invertible valuation ideal.
\end{enumerate}
\end{theorem}

\begin{proof}
(1) $\Rightarrow$ (2) This is clear.

(2) $\Rightarrow$ (8) $\Rightarrow$ (3) $\Rightarrow$ (1) This follows from Theorem~\ref{Theorem 3.15}.

(3) $\Leftrightarrow$ (4) See \cite[Corollary 2.9]{acz13}.

(4) $\Rightarrow$ (5) Note that $D[X]$ is a P$v$MD, so $D$ is a P$v$MD \cite[Theorem 3.7]{k89} and $D[X]$ is $t$-treed. Thus, the result follows by the equivalence of (2) and (3) above.

(5) $\Rightarrow$ (6) ${\rm Cl}_t(D[X]_{N_v})=\{0\}$ by \cite[Theorem 2.14, Proposition 2.1 and Corollary 2.3]{k89}. Thus, $D[X]_{N_v}$ is a VFD by Propositions~\ref{Proposition 3.3} and \ref{Proposition 4.3}(2).

(6) $\Rightarrow$ (7) Since $D$ is a UMT-domain, $D[X]_{N_v}$ is treed by Lemma~\ref{Lemma 4.4}. Thus, $D[X]_{N_v}$ is an independent ring of Krull type \cite[Theorem 3.4]{cr20}. Moreover, since each maximal ideal of $D[X]_{N_v}$ is a $t$-ideal, the result follows.

(7) $\Rightarrow$ (4) See \cite[Lemma 2.2]{cc23}.
\end{proof}

We say that $D$ is a {\em generalized Krull domain} (in the sense of \cite{p72}) if $D$ is a P$v$MD such that $t$-$\dim(D)\leq 1$ and $D$ is of finite $t$-character. Recall that $t$-$\dim(D)=1$ if and only if $D$ is not a field and each maximal $t$-ideal of $D$ is a height-one prime ideal, so an integral domain of $t$-dimension one is $t$-treed. The class of integral domains of $t$-dimension at most one includes Krull domains, generalized Krull domains, and one-dimensional integral domains.

\begin{corollary}\label{Corollary 4.6}
Let $D$ be an integral domain with $t$-$\dim(D)=1$. The following statements are equivalent\textnormal{:}
\begin{enumerate}
\item Each nonzero $t$-ideal of $D$ is a finite $t$-product of valuation ideals.
\item $D$ is a $t$-VIFD.
\item $D$ is a generalized Krull domain.
\item $D$ is a ring of Krull type.
\end{enumerate}
\end{corollary}

\begin{proof}
It is clear that if $t$-$\dim(D)=1$, then $D$ is an independent ring of Krull type if and only if $D$ is a generalized Krull domain, if and only if $D$ is a ring of Krull type. Thus, the result follows directly from Theorem~\ref{Theorem 4.5}.
\end{proof}

Let $D(X)=\{\frac{f}{g}\mid f,g\in D[X]$ and $c(g)=D\}$. Then $D(X)$, called the {\em Nagata ring} of $D$, is a ring such that $D[X]\subseteq D(X)\subseteq D[X]_{N_v}$. It is known that each nonzero prime ideal of a treed domain is a $t$-ideal. Hence, a treed domain is a $t$-treed domain whose nonzero maximal ideals are $t$-ideals. We next characterize when a treed domain is a VIFD.

\begin{corollary}\label{Corollary 4.7}
Let $D$ be a treed domain. The following statements are equivalent\textnormal{:}
\begin{enumerate}
\item $D$ is an ${\rm h}$-local Pr\"ufer domain.
\item $D$ is a VIFD.
\item Each nonzero ideal of $D$ is a finite product of valuation ideals.
\item Every nonzero principal ideal of $D$ can be written as a finite product of comaximal valuation ideals.
\item Every nonzero principal ideal of $D$ can be written as a finite intersection of comaximal valuation ideals.
\item Every nonzero prime ideal of $D$ contains an invertible valuation ideal.
\item $D(X)$ is a VFD and $D$ is a UMT-domain.
\end{enumerate}
\end{corollary}

\begin{proof}
(1) $\Leftrightarrow$ (2) $\Leftrightarrow$ (3) $\Leftrightarrow$ (4) $\Leftrightarrow$ (6) These follow directly from Theorem~\ref{Theorem 3.15}.

(1) $\Leftrightarrow$ (7) Since $D$ is treed, every nonzero prime of $D$ is a $t$-ideal. Moreover, if each maximal ideal of $D$ is a $t$-ideal, then (i) $D$ is an independent ring of Krull type if and only if $D$ is an ${\rm h}$-local Pr\"ufer domain and (ii) $D(X)=D[X]_{N_v}$. Thus, the result is an immediate consequence of Theorem~\ref{Theorem 4.5}.

(4) $\Leftrightarrow$ (5) This follows from the fact that if $I$ and $J$ are comaximal ideals of $D$, then $IJ=I\cap J$.
\end{proof}

Next, we provide a variant of \cite[Proposition 4.6]{cr20}. Recall that $D$ satisfies the {\rm principal ideal theorem} if every minimal prime ideal of a nonzero principal ideal of $D$ is of height one. For example, Noetherian domains and Krull domains satisfy the principal ideal theorem, while a VFD does (in general) not satisfy the principal ideal theorem.

\begin{theorem}\label{Theorem 4.8}
Let $D$ be a $t$-VIFD. The following statements are equivalent\textnormal{:}
\begin{enumerate}
\item $D$ is a generalized Krull domain.
\item $t$-$\dim(D)\leq 1$.
\item $D$ satisfies the principal ideal theorem.
\item $D$ is completely integrally closed.
\item $\bigcap_{n\in\mathbb{N}} (I^n)_t=(0)$ for each proper $t$-invertible $t$-ideal $I$ of $D$.
\end{enumerate}
\end{theorem}

\begin{proof}
(1) $\Rightarrow$ (2), (4) It is obvious that $t$-$\dim(D)\leq 1$. If $M\in t$-${\rm Max}(D)$, then $D_M$ is a valuation domain of dimension $\leq 1$, and hence $D_M$ is completely integrally closed. Therefore, $D=\bigcap_{M\in t\textnormal{-}{\rm Max}(D)} D_M$ is completely integrally closed.

(2) $\Rightarrow$ (3) This is obvious.

(3) $\Rightarrow$ (1) We have to show that $D$ is a P$v$MD, $D$ is of finite $t$-character and $t$-$\dim(D)\leq 1$. It follows from Corollary~\ref{Corollary 3.13} that $D$ is of finite $t$-character. Let $M\in t$-${\rm Max}(D)$. Observe that $D_M$ satisfies the principal ideal theorem. Moreover, $D_M$ is a VFD by Proposition~\ref{Proposition 3.10}. Consequently, $D_M$ is a P$v$MD by \cite[Proposition 4.6]{cr20}. It follows from \cite[Theorem 4.1(2)]{efz16} that $D$ is a P$v$MD. Finally, let $P$ be a nonzero prime $t$-ideal of $D$. Since $D$ satisfies the principal ideal theorem, $P$ is the union of all height-one prime ideals of $D$ that are contained in $P$. Since $D$ is $t$-treed, there is precisely one height-one prime ideal of $D$ that is contained in $P$, and hence $P$ is a height-one prime ideal of $D$. Thus, $t$-$\dim(D)\leq 1$.

(4) $\Rightarrow$ (5) Let $I$ be a proper $t$-invertible $t$-ideal of $D$ and set $J=\bigcap_{n\in\mathbb{N}} (I^n)_t$. Observe that $I^{-1}J\subseteq J$, and thus $I^{-1}\subseteq (J:J)$. Assume that $J\not=(0)$. Then $(J:J)=D$ (since $D$ is completely integrally closed), and hence $I=I_v=I^{-1}=D$, a contradiction.

(5) $\Rightarrow$ (3) Let $P$ be a minimal prime ideal of a nonzero principal ideal of $D$. Since $D$ is a $t$-VIFD, there exists a $t$-invertible valuation $t$-ideal $I$ of $D$ such that $P$ is a minimal prime ideal of $I$. Let $Q$ be a prime ideal of $D$ such that $Q\subsetneq P$. It suffices to show that $Q=(0)$. Assume that $Q\not=(0)$. There exists some $x\in Q\setminus\{0\}$. We infer that $xD\subseteq J\subseteq Q$ for some $t$-invertible valuation $t$-ideal $J$ of $D$. Observe that $\sqrt{J}\subseteq Q\subsetneq P=\sqrt{I}$. Therefore, $x\in xD\subseteq J\subseteq\bigcap_{n\in\mathbb{N}} (I^n)_t=(0)$ by Corollary~\ref{Corollary 2.10}(4), a contradiction.
\end{proof}

\section{Almost valuation factorization domains}

We will say that $D$ is an {\em almost valuation factorization domain} (AVFD) if for each nonzero nonunit $a\in D$, there is an $n\in\mathbb{N}$ such that $a^n$ can be written as a finite product of valuation elements. Clearly, a VFD is an AVFD and a $t$-VIFD with torsion $t$-class group is an AVFD (see Proposition~\ref{Proposition 5.5}).

\begin{lemma}\label{Lemma 5.1}
Let $D$ be an integral domain and let $b\in D$ be a nonzero nonunit that can be written as a finite product of valuation elements of $D$\textnormal{:}
\begin{enumerate}
\item $\min\{k\in\mathbb{N}\mid b$ is a product of $k$ valuation elements of $D\}=|\mathcal{P}(bD)|$.
\item There are valuation elements $(y_P)_{P\in\mathcal{P}(bD)}$ of $D$ such that $b=\prod_{P\in\mathcal{P}(bD)} y_P$ and $\sqrt{y_QD}=Q$ for each $Q\in\mathcal{P}(bD)$.
\end{enumerate}
\end{lemma}

\begin{proof}
This can be proved along the same lines as \cite[Lemma 1.12 and Proposition 1.13]{cr20}.
\end{proof}

For the next result we mimic the proof of \cite[Proposition 2.1]{cr20} that a VFD is a Schreier domain.

\begin{proposition}\label{Proposition 5.2}
Let $D$ be an integral domain and let $\Omega$ be the set of all finite products of units and valuation elements of $D$. Then for each $x,y,z\in\Omega$ with $x\mid_D yz$, there are some $a,b\in\Omega$ such that $x=ab$, $a\mid_D y$, $b\mid_D z$ and $\frac{y}{a},\frac{z}{b}\in\Omega$.
\end{proposition}

\begin{proof}
If $r,s\in\Omega$, then we write $r\mid_{\Omega} s$ if $r\mid_D s$ and $\frac{s}{r}\in\Omega$. Note that $\Omega$ consists precisely of the units and the finite nonempty products of valuation elements of $D$. Therefore, it is sufficient to show by induction that for all $n\in\mathbb{N}$, $x\in D$ and $y,z\in\Omega$ such that $x$ is a product of $n$ valuation elements of $D$ and $x\mid_D yz$, there are some $a,b\in\Omega$ such that $x=ab$, $a\mid_{\Omega} y$ and $b\mid_{\Omega} z$.

\medskip
First let $n=1$. By Lemma~\ref{Lemma 5.1}(1), it is sufficient to show by induction that for all $k\in\mathbb{N}$, for each valuation element $x\in D$ and for all $y,z\in\Omega$ such that $x\mid_D yz$ and $|\mathcal{P}(yD)|+|\mathcal{P}(zD)|=k$, there are some $a,b\in\Omega$ such that $x=ab$, $a\mid_{\Omega} y$ and $b\mid_{\Omega} z$. Let $k\in\mathbb{N}$, let $x\in D$ be a valuation element and let $y,z\in\Omega$ be such that $x\mid_D yz$ and $|\mathcal{P}(yD)|+|\mathcal{P}(zD)|=k$. Since $yz\in xD\subseteq\sqrt{xD}\in {\rm Spec}(D)$ by \cite[Proposition 1.7(1)]{cr20}, we have that $y\in\sqrt{xD}$ or $z\in\sqrt{xD}$. Without restriction let $y\in\sqrt{xD}$. By Lemma~\ref{Lemma 5.1}(2) there are valuation elements $(y_P)_{P\in\mathcal{P}(yD)}$ of $D$ such that $y=\prod_{P\in\mathcal{P}(yD)} y_P$ and $\sqrt{y_QD}=Q$ for each $Q\in\mathcal{P}(yD)$. Consequently, there is some $Q\in\mathcal{P}(yD)$ such that $y_Q\in\sqrt{xD}$, and hence $\sqrt{y_QD}\subseteq\sqrt{xD}$. We infer by \cite[Corollary 1.2(1)]{cr20} that $y_QD\subseteq xD$ or $xD\subseteq y_QD$. Set $y^{\prime}=\prod_{P\in\mathcal{P}(yD)\setminus\{Q\}} y_P$.

\medskip
{\noindent}{{\bf Case 1.}} $y_QD\subseteq xD$. Then $y_Q=xu$ for some $u\in D$. It follows from \cite[Proposition 1.1(3)]{cr20} that $u$ is a unit or a valuation element of $D$. Set $a=x$ and $b=1$. Obviously, $a,b\in\Omega$, $x=ab$ and $b\mid_{\Omega} z$. Moreover, $a\mid_D y_Q\mid_D y$ and $\frac{y}{a}=uy^{\prime}\in\Omega$, and thus $a\mid_{\Omega} y$.

\medskip
{\noindent}{{\bf Case 2.}} $xD\subsetneq y_QD$. Then $x=y_Qw$ for some nonunit $w\in D$. We infer by \cite[Proposition 1.1(3)]{cr20} that $w$ is a valuation element of $D$. Observe that $y_Qw=x\mid_D yz=y_Qy^{\prime}z$, and thus $w\mid_D y^{\prime}z$. Furthermore, $\mathcal{P}(y^{\prime}D)=\mathcal{P}(yD)\setminus\{Q\}$, and hence $|\mathcal{P}(y^{\prime}D)|+|\mathcal{P}(zD)|<k$. We infer by the induction hypothesis that there are some $a^{\prime},b\in\Omega$ such that $w=a^{\prime}b$, $a^{\prime}\mid_{\Omega} y^{\prime}$ and $b\mid_{\Omega} z$. Set $a=y_Qa^{\prime}$. Then $a\in\Omega$, $x=y_Qw=y_Qa^{\prime}b=ab$ and $a\mid_{\Omega} y$. This concludes the proof of the base step.

\medskip
Now let $n\in\mathbb{N}$, $x\in D$ and $y,z\in\Omega$ be such that $x$ is a product of $n+1$ valuation elements of $D$ and $x\mid_D yz$. Then $x=\prod_{i=1}^{n+1} v_i$ for some valuation elements $v_i$ of $D$. Set $x^{\prime}=\prod_{i=1}^n v_i$. Then $x^{\prime}\mid_D yz$. We infer by the induction hypothesis that there are some $a^{\prime},b^{\prime}\in\Omega$ such that $x^{\prime}=a^{\prime}b^{\prime}$, $a^{\prime}\mid_{\Omega} y$ and $b^{\prime}\mid_{\Omega} z$. Since $a^{\prime}b^{\prime}v_{n+1}=x\mid_D yz$, we infer that $v_{n+1}\mid_D\frac{y}{a^{\prime}}\frac{z}{b^{\prime}}$. Since $v_{n+1}$ is a valuation element of $D$ and $\frac{y}{a^{\prime}},\frac{z}{b^{\prime}}\in\Omega$, it follows by the base step that there are some $a^{\prime\prime},b^{\prime\prime}\in\Omega$ such that $v_{n+1}=a^{\prime\prime}b^{\prime\prime}$, $a^{\prime\prime}\mid_{\Omega}\frac{y}{a^{\prime}}$ and $b^{\prime\prime}\mid_{\Omega}\frac{z}{b^{\prime}}$. Set $a=a^{\prime}a^{\prime\prime}$ and $b=b^{\prime}b^{\prime\prime}$. Then $a,b\in\Omega$, $x=x^{\prime}v_{n+1}=a^{\prime}b^{\prime}a^{\prime\prime}b^{\prime\prime}=ab$, $a\mid_{\Omega} y$ and $b\mid_{\Omega} z$.
\end{proof}

We say that $D$ is an {\em almost Schreier domain} if for all nonzero $x,y,z\in D$ with $x\mid_D yz$ there are $n\in\mathbb{N}$ and $a,b\in D$ such that $x^n=ab$, $a\mid_D y^n$ and $b\mid_D z^n$ \cite{dk10}. It is known that if $D$ is an almost Schreier domain, then ${\rm Cl}_t(D)$ is a torsion group \cite[Theorem 3.1]{dk10} and if $D$ is an integrally closed almost Schreier domain, then $D[X]$ is an almost Schreier domain \cite[Theorem 4.4]{dk10}.

\begin{theorem}\label{Theorem 5.3}
Let $D$ be an integral domain. Then $D$ is an AVFD if and only if $D$ is an almost Schreier domain and every nonzero prime $t$-ideal of $D$ contains a valuation element of $D$.
\end{theorem}

\begin{proof}
Let $\Omega$ be the set of all finite products of units and valuation elements of $D$.

($\Rightarrow$) Let $D$ be an AVFD. Let $x,y,z\in D$ be nonzero such that $x\mid_D yz$. Since $D$ is an AVFD, there are $r,s,t\in\mathbb{N}$ such that $x^r,y^s,z^t\in\Omega$. Set $n=rst$. Then $n\in\mathbb{N}$, $x^n,y^n,z^n\in\Omega$ and $x^n\mid_D y^nz^n$. We infer by Proposition~\ref{Proposition 5.2} that there are $a,b\in D$ such that $x^n=ab$, $a\mid_D y^n$ and $b\mid_D z^n$. Therefore, $D$ is an almost Schreier domain. Now let $P$ be a nonzero prime $t$-ideal of $D$. Choose a nonzero $a\in P$. Then there exists $n\in\mathbb{N}$ such that $a^n$ is a finite product of valuation elements of $D$. Note that $a^n\in P$. Thus, $P$ contains a valuation element of $D$.

($\Leftarrow$) Let $D$ be an almost Schreier domain for which every nonzero prime $t$-ideal contains a valuation element of $D$. Let $\Sigma=\{a\in D\mid a^n\in\Omega$ for some $n\in\mathbb{N}\}$. Clearly, $\Omega$ and $\Sigma$ are multiplicatively closed subsets of $D$. We show that $\Sigma$ is a divisor-closed subset of $D$. Let $a\in\Sigma$ and $b\in D$ be such that $b\mid_D a$. There are some $n,m\in\mathbb{N}$ and $v_i\in D$ such that $a^n=\prod_{i=1}^m v_i$ and $v_i$ is a unit or a valuation element of $D$ for each $i\in [1,m]$. We have that $b^n\mid_D\prod_{i=1}^m v_i$. Since $D$ is an almost Schreier domain, it follows by induction that there are some $k\in\mathbb{N}$ and $w_i\in D$ such that $b^{nk}=\prod_{i=1}^m w_i$ and $w_j\mid_D v_j^k$ for each $j\in [1,m]$. Let $j\in [1,m]$. If $v_j$ is a unit of $D$, then $w_j$ is a unit of $D$. Now let $v_j$ be a valuation element of $D$. Since $\sqrt{v_jD}\subseteq\sqrt{w_jD}$, we have that $w_j$ is a unit or a valuation element of $D$ by \cite[Proposition 1.1(3)]{cr20}. Consequently, $b^{nk}\in\Omega$, and thus $b\in\Sigma$.

It remains to show that $D\setminus\{0\}\subseteq\Sigma$. Assume that there is some $z\in D\setminus (\Sigma\cup\{0\})$. Since $\Sigma$ is a divisor-closed subset of $D$ (as shown before), we have that $zD\cap\Sigma=\emptyset$. Therefore, there exists a prime $t$-ideal $P$ of $D$ such that $zD\subseteq P$ and $P\cap\Sigma=\emptyset$. Since $z\in P$, we have that $P$ is nonzero, and thus $P$ contains a valuation element of $D$. This implies that $\emptyset\not=P\cap\Omega\subseteq P\cap\Sigma=\emptyset$, a contradiction.
\end{proof}

\begin{corollary}\label{Corollary 5.4}
Let $D$ be an AVFD. Then $D$ is an integrally closed weakly Matlis domain and ${\rm Cl}_t(D)$ is a torsion group.
\end{corollary}

\begin{proof}
It is an immediate consequence of Proposition~\ref{Proposition 3.12} and Theorem~\ref{Theorem 5.3} that $D$ is an integrally closed weakly Matlis almost Schreier domain. We infer by \cite[Theorem 3.1]{dk10} that ${\rm Cl}_t(D)$ is a torsion group.
\end{proof}

Recall that (i) VFDs are AVFDs by definition and (ii) $D$ is a VFD if and only if $D$ is a VIFD with ${\rm Pic}(D)=\{0\}$, if and only if $D$ is a $w$-VIFD with ${\rm Cl}_t(D)=\{0\}$ by Corollary~\ref{Corollary 3.6}. The next result also shows that a VIFD $D$ for which ${\rm Pic}(D)$ is a torsion group is an AVFD.

\begin{proposition}\label{Proposition 5.5}
Let $D$ be an integral domain and let $*$ be a star operation of finite type on $D$ such that $D$ is a $*$-VIFD and ${\rm Cl}_*(D)$ is a torsion group. Then $D$ is an AVFD.
\end{proposition}

\begin{proof}
Let $a\in D$ be a nonzero nonunit. Then $aD=(\prod_{i=1}^n I_i)_*$ for some $n\in\mathbb{N}$ and proper valuation $*$-ideals $I_i$ of $D$ by Proposition~\ref{Proposition 3.3}. Clearly, each $I_i$ is $*$-invertible, so by assumption that ${\rm Cl}_*(D)$ is a torsion group, there is an $m\in\mathbb{N}$ such that for each $i\in [1,n]$, there is some $a_i\in D$ such that $(I_i^m)_*=a_iD$. Then each $a_i$ is a valuation element by Proposition~\ref{Proposition 2.9}. Thus, $a^m$ can be written as a finite product of valuation elements.
\end{proof}

We say that $D$ is an {\em almost GCD domain} (AGCD domain) if for each $a,b\in D$, there is an $n\in\mathbb{N}$ such that $a^nD\cap b^nD$ is principal \cite{z85}. It is known that an integrally closed domain $D$ is an AGCD domain if and only if $D$ is a P$v$MD such that ${\rm Cl}_t(D)$ is a torsion group \cite[Corollary 3.8 and Theorem 3.9]{z85} and an AGCD domain is an almost Schreier domain \cite[Proposition 2.2]{dk10}.

\begin{corollary}\label{Corollary 5.6}
Let $D$ be a $t$-treed domain. The following statements are equivalent\textnormal{:}
\begin{enumerate}
\item $D$ is an AVFD.
\item $D$ is an integrally closed weakly Matlis AGCD domain.
\item $D$ is an independent ring of Krull type such that ${\rm Cl}_t(D)$ is a torsion group.
\item ${\rm Cl}_t(D)$ is a torsion group and each nonzero prime ideal of $D$ contains a valuation element.
\item $D$ is a $t$-VIFD such that ${\rm Cl}_t(D)$ is a torsion group.
\item $D[X]$ is an AVFD.
\end{enumerate}
\end{corollary}

\begin{proof}
(1) $\Rightarrow$ (4) It is clear that each nonzero prime ideal of $D$ contains a valuation element. Moreover, ${\rm Cl}_t(D)$ is a torsion group by Corollary~\ref{Corollary 5.4}.

(2) $\Leftrightarrow$ (3) See, for example, \cite[Corollary 3.8 and Theorem 3.9]{z85}.

(3) $\Rightarrow$ (1) Note that $D$ is a $t$-VIFD by Theorem~\ref{Theorem 4.5}. Thus, by Proposition~\ref{Proposition 5.5}, $D$ is an AVFD.

(3) $\Rightarrow$ (6) Observe that $D$ is a P$v$MD with ${\rm Cl}_t(D)$ torsion. Hence, ${\rm Cl}_t(D[X])$ is a torsion group \cite[Theorem 5.6]{z85}. Moreover, $D[X]$ is a $t$-VIFD by Theorem~\ref{Theorem 4.5}. Thus, $D[X]$ is an AVFD by Proposition~\ref{Proposition 5.5}.

(4) $\Rightarrow$ (3), (5) This follows from Theorem~\ref{Theorem 4.5}.

(5) $\Rightarrow$ (1) This is an immediate consequence of Proposition~\ref{Proposition 5.5}.

(6) $\Rightarrow$ (1) Let $D[X]$ be an AVFD. Let $a$ be a nonzero nonunit of $D$. Then $a$ is a nonzero nonunit of $D[X]$, and hence there are some $k,n\in\mathbb{N}$ and valuation elements $v_i$ of $D[X]$ such that $a^k=\prod_{i=1}^n v_i$. Observe that $v_i\in D$ for each $i\in [1,n]$, and hence $v_i$ is a valuation element of $D$ for each $i\in [1,n]$ by \cite[Lemma 2.5]{cr20}. Therefore, $D$ is an AVFD.
\end{proof}

\begin{corollary}\label{Corollary 5.7}
Let $D$ be an integral domain and let $D[X]$ be the polynomial ring over $D$. Then $D[X]$ is an AVFD if and only if $D$ is an AVFD and every upper to zero in $D[X]$ contains a valuation element of $D[X]$.
\end{corollary}

\begin{proof} ($\Rightarrow$) Let $D[X]$ be an AVFD. Then $D$ is an AVFD by Corollary~\ref{Corollary 5.6} and every upper to zero in $D[X]$ contains a valuation element of $D[X]$ by Theorem~\ref{Theorem 5.3}.

($\Leftarrow$) Now let $D$ be an AVFD such that every upper to zero in $D[X]$ contains a valuation element of $D[X]$. It follows by Theorem~\ref{Theorem 5.3} and Corollary~\ref{Corollary 5.4} that $D$ is an integrally closed almost Schreier domain and each nonzero prime ($t$-)ideal of $D$ contains a valuation element of $D$. We infer by \cite[Theorem 4.4]{dk10} that $D[X]$ is an almost Schreier domain. Now let $Q$ be a nonzero prime $t$-ideal of $D[X]$. If $Q\cap D=(0)$, then $Q$ contains a valuation element of $D[X]$ by assumption. Now let $Q\cap D\not=(0)$ and set $P=Q\cap D$. Then $P$ is a nonzero prime ideal of $D$, and thus $P$ contains a valuation element of $D$. Consequently, $Q$ contains a valuation element of $D$. It follows from \cite[Lemma 2.5]{cr20} that $Q$ contains a valuation element of $D[X]$. Therefore, $D[X]$ is an AVFD by Theorem~\ref{Theorem 5.3}.
\end{proof}

\begin{example}\label{Example 5.8}
{\em Let $L$ be an algebraic number field and let $D$ be an order in $L$. Then $D$ is an almost Schreier domain \cite[Remark 6.4]{dk10} and if $D$ is the principal order in $L$ $($i.e., $D$ is integrally closed$\,)$, then $D$ is an AVFD by Corollary~\ref{Corollary 5.6}.}
\end{example}

It is known that an atomic VFD is a UFD \cite[Corollary 2.4]{cr20}, and hence it satisfies the principal ideal theorem. The next result shows that this is true for AVFDs.

\begin{proposition}\label{Proposition 5.9}
Let $D$ be an atomic domain\textnormal{:}
\begin{enumerate}
\item If $D$ is an AVFD, then $D$ satisfies the principal ideal theorem.
\item If $D$ is a $t$-VIFD such that ${\rm Cl}_t(D)$ is a torsion group, then $D$ is a generalized Krull domain.
\end{enumerate}
\end{proposition}

\begin{proof}
(1) Let $D$ be an AVFD, let $x\in D$ be nonzero and let $P\in\mathcal{P}(xD)$. There exists an $n\in\mathbb{N}$ such that $x^n$ is a finite product of valuation elements of $D$. Note that $P\in\mathcal{P}(x^nD)$, and hence $P\in\mathcal{P}(vD)$ for some valuation element $v\in D$. This implies that $P=\sqrt{vD}$ by Proposition~\ref{Proposition 2.3}(2). Assume that $P$ is not a height-one prime ideal of $D$. Then there is a nonzero prime ideal $Q$ of $D$ such that $Q\subsetneq P$. Note that $Q$ contains a valuation element $a\in D$. Clearly, there is an atom $u\in Q$ such that $aD\subseteq uD$. We infer by \cite[Proposition 1.1(2)]{cr20} that $u$ is a valuation element of $D$. Moreover, $\sqrt{uD}\subseteq Q\subsetneq P=\sqrt{vD}$, and hence $uD\subsetneq vD\subsetneq D$ \cite[Corollary 1.2(1)]{cr20}. This contradicts the fact that $u$ is an atom of $D$.

(2) Let $D$ be a $t$-VIFD such that ${\rm Cl}_t(D)$ is a torsion group. Then $D$ is an AVFD by Proposition~\ref{Proposition 5.5}. Consequently, $D$ satisfies the principal ideal theorem by (1), and hence $D$ is a generalized Krull domain by Theorem~\ref{Theorem 4.8}.
\end{proof}

The {\em exponent} of a group $G$, denoted ${\rm exp}(G)$, is defined by $\inf\{k\in\mathbb{N}\mid x^k=1$ for all $x\in G\}$. We proceed with providing a partial generalization of \cite[Corollary 2.4]{cr20} that characterizes when a VFD is atomic. Let $\mathbb{N}_{\geq m}=\{x\in\mathbb{N}\mid x\geq m\}$ for each $m\in\mathbb{N}$.

\begin{theorem}\label{Theorem 5.10}
Let $D$ be an integral domain such that ${\rm exp}({\rm Cl}_t(D))$ is finite. Then $D$ is a Krull domain if and only if $D$ is an atomic $t$-VIFD, and in this case, $D$ is an AVFD.
\end{theorem}

\begin{proof}
Clearly, every Krull domain is an atomic $t$-VIFD. Conversely, suppose that $D$ is an atomic $t$-VIFD. It follows from Proposition~\ref{Proposition 5.9}(2) that $D$ is a generalized Krull domain. Therefore, $D$ satisfies the ACCP by Proposition~\ref{Proposition 3.18}.

Next we show that $D$ satisfies the ACC on $t$-invertible $t$-ideals of $D$. Let $(I_i)_{i\in\mathbb{N}}$ be an ascending sequence of $t$-invertible $t$-ideals of $D$. Let $n$ be the exponent of ${\rm Cl}_t(D)$. Observe that $(I_i^n)_t$ is a principal ideal of $D$ for each $i\in\mathbb{N}$. Since $D$ satisfies the ACCP, there exists some $m\in\mathbb{N}$ such that $(I_m^n)_t=(I_k^n)_t$ for all $k\in\mathbb{N}_{\geq m}$. Let $k\in\mathbb{N}_{\geq m}$ and let $M\in t$-${\rm Max}(D)$. There are nonzero $a,b\in D$ such that $(I_m)_M=aD_M$ and $(I_k)_M=bD_M$. It follows that $a^nD_M=((I_m^n)_t)_M=((I_k^n)_t)_M=b^nD_M$, and thus $(\frac{a}{b})^n$ is a unit of $D_M$. Note that $D_M$ is integrally closed by Corollary~\ref{Corollary 3.13}. Consequently, $\frac{a}{b}\in D_M$, and hence $\frac{a}{b}$ is a unit of $D_M$. This implies that $(I_m)_M=aD_M=bD_M=(I_k)_M$. Therefore, $I_m=I_k$.

It remains to show that every nonzero $t$-ideal of $D$ is $t$-invertible. Let $I$ be a nonzero $t$-ideal of $D$. Then there exists a maximal element $J$ of the set of all $t$-invertible $t$-ideals of $D$ that are contained in $I$. Assume that $J\subsetneq I$. Then there exists some $a\in I\setminus J$. Since $D$ is a P$v$MD, we infer that $(J+aD)_t$ is a $t$-invertible $t$-ideal of $D$. Furthermore, $J\subsetneq (J+aD)_t\subseteq I$, a contradiction. Therefore, $I=J$ is a $t$-invertible $t$-ideal of $D$.

It is obvious that if ${\rm exp}({\rm Cl}_t(D))$ is finite, then ${\rm Cl}_t(D)$ is a torsion group. Hence, in this case, $D$ is an AVFD by Proposition~\ref{Proposition 5.5}.
\end{proof}

We end this paper with an example which shows that the assumption that ${\rm exp}({\rm Cl}_t(D))$ is finite in Theorem~\ref{Theorem 5.10} is crucial.

\begin{example}\label{Example 5.11}
{\em Let $D$ be a one-dimensional atomic Pr\"ufer domain of finite character such that $D$ is not a Dedekind domain. $($For an example of such a domain see \cite{g74}.$)$ Then $D$ is an atomic $t$-VIFD that is not a Krull domain.}
\end{example}

\bigskip
\noindent
\textbf{Acknowledgements.} We would like to thank the referee for various comments and suggestions that improved this paper.

\end{document}